\documentclass{article}
\usepackage{amsmath,amssymb,enumerate,amsthm}

\usepackage{color}
\usepackage{amsmath,amsfonts,calrsfs,amssymb,color,verbatim,eucal,mathrsfs}

\newtheorem{theorem}{Theorem}[section]

\newtheorem{corollary}[theorem]{Corollary}

\newtheorem{proposition}[theorem]{Proposition}
\newtheorem{definition}[theorem]{Definition}
\newtheorem{remark}[theorem]{Remark}

\newtheorem{hyp}{Hypotheses}{\rm}

\allowdisplaybreaks

\newcommand{\R}{{\mathbb R}}

\newcommand{\N}{{\mathbb N}}

\newcommand{\Om}{{\Omega}}

\newcommand{\F}{\mathcal F}

\newcommand{\ve}{\varepsilon}
\newcommand{\eps}{\varepsilon}
\newcommand{\ra}{\rightarrow}

\renewcommand{\tilde}[1]{\widetilde{#1}}

\newcommand{\set}[1]{{\left\{#1\right\}}}
\newcommand{\pa}[1]{{\left(#1\right)}}

\newcommand{\gen}[1]{{\left\langle #1\right\rangle}}
\newcommand{\abs}[1]{{\left|#1\right|}}
\newcommand{\norm}[1]{{\left\|#1\right\|}}

\newcommand{\eqsys}[1]{{\left\{\begin{array}{ll}#1\end{array}\right.}}
\newcommand{\tc}{\, \middle |\,}

\title{On functions of bounded variation on convex domains in Hilbert spaces
\thanks{The authors are members of G.N.A.M.P.A. of the Italian Istituto Nazionale di Alta Matematica (INdAM).
Work partially supported by the INdAM-GNAMPA Project 2019 ``Metodi analitici per lo studio di PDE e problemi collegati in 
dimensione infinita'' and by the Research project PRIN 2015 MIUR 2015233N5A ``Deterministic and stochastic evolution 
equations" .}
}

\author{L. Angiuli\thanks{Dipartimento di Matematica e Fisica ``Ennio De Giorgi", Universit\`a del Salento. 
POB 193, 73100 Lecce, Italy; e-mail luciana.angiuli}, 
S. Ferrari\thanks{Dipartimento di Scienze Matematiche, Fisiche e Informatiche, 
Universit\`a di Parma, Parco Area delle Scienze 53/A, 43124 Parma, Italy; 
e-mail simone.ferrari1@unipr.it}  \, and 
D. Pallara\thanks{Dipartimento di Matematica e Fisica ``Ennio De Giorgi", Universit\`a del Salento, and INFN, 
Sezione di Lecce. POB 193, 73100 Lecce, Italy; e-mail diego.pallara@unisalento.it}}
\date{}
\begin{document}

\maketitle

%%%% Abstract text to be placed here %%%%%%%%%%%%
\begin{abstract}
We study functions of bounded variation (and sets of finite perimeter) on a convex open set $\Om\subseteq X$, $X$ being 
an infinite dimensional real Hilbert space. We relate the total variation of such functions, defined through an integration 
by parts formula, to the short-time behaviour of the semigroup associated with a perturbation of the 
Ornstein-Uhlenbeck operator.
\end{abstract}
%%%%%%%%%%%%%%%%%%%%%%%%%%%

%%%%%%%%%% Insert the texts which can accomdate on firstpage in the tag "fmtext" %%%%%

\section*{Introduction}
In this paper we study some properties of functions of bounded variation ($BV$ functions, for short) defined on an open convex subset of a real separable Hilbert space, endowed with a weighted Gaussian measure.

In finite dimension the theory of $BV$ functions is widely developed (see e.g. \cite{AFP00} and the references therein), whereas in the infinite dimensional setting the analysis is still at the initial stage and many basic properties are  unexplored. Besides the interest on its own, the study of $BV$ functions in infinite dimensional spaces is motivated by problems arising in calculus of variations, stochastic analysis and connected with the applications in information technology (see, for example, \cite{HinHuc,Don,Roc2Zhu,Tre,Zam}).

$BV$ functions for Gaussian measures in separable Banach spaces were introduced in \cite{Fuk} using
Dirichlet forms. Inspired by the results in finite dimension, which connect the theory of functions of bounded variation to that of semigroups of bounded operators, the authors  of \cite{FukHin} have proved an elegant characterisation of $BV$ functions in terms of the short-time behaviour of the Ornstein--Uhlenbeck semigroup. More precisely, in a separable Banach space $X$, if $\gamma$ is a centered and nondegenerate Gaussian measure on $X$ and $u$ belongs to the Orlicz space $L(\log L)^{1/2}(X,\gamma)$, then $u\in BV(X, \gamma)$ if, and only if,
\begin{equation*}
\liminf_{t \to 0^+}\int_X |D_HS(t)u|_Hd\gamma<+\infty ,
\end{equation*}
where $D_H$ is the gradient operator along the Cameron--Martin space $H$ (see Section \ref{prel}) and $S(t)$ is the classical Ornstein--Uhlenbeck semigroup defined via the Mehler formula (see \eqref{Mehler}). This latter is the analogous, in the Gaussian setting, of the heat semigroup used by De Giorgi in \cite{Deg} to provide the original definition of $BV$ functions in the Euclidean case. An analytic approach based on geometric measure theory is proposed in \cite{AMMP} to prove, as in the finite dimensional case, the equivalence of different definitions of $BV(X,\gamma)$ functions also, as in \cite{FukHin}, in terms of the Ornstein--Uhlenbeck semigroup $S(t)$ near $t=0$. Similar De Giorgi-type characterisations of $BV$ functions have been obtained for weighted Gaussian measures and more recently for general Fomin differentiable measures in Hilbert spaces, see \cite{DPL18} and the reference therein.

Beside the difficulty of considering general measures, another difficulty of different nature comes from the consideration of functions defined in domains rather than in the whole space. These difficulties come from the lack of factorisation of the underlying measure (that is lost even for Gaussian measures
in domains) and the unavailability of decomposition of the domain through the classical method of local charts. Therefore,
the easiest interesting case seems to be that of {\em convex} domains, that are possible to deal with through {\em global}
penalisation techniques. This is the approach we followed in \cite{AFP18} (see also \cite{LMP})
and in this paper we take advantage of the results proved there. We start from a weighted Gaussian measure 
$\nu:=e^{-U}\gamma$ in a Hilbert space $X$, where $U:X\to\R$ is convex
and sufficiently regular, and consider an open convex domain $\Omega\subseteq X$. After introducing the Cameron--Martin space
$H$ and the Malliavin gradient $D_H$ along it, we define the form
$(u,v)\mapsto \int_\Omega\langle D_Hu,D_Hv\rangle_Hd\nu$
on the appropriate Sobolev spaces. The perturbed Ornstein--Uhlenbeck operator $L_\Omega$ is then defined in the usual
variational way, and it is the generator of an analytic, strongly continuous and contraction semigroup $T_\Omega(t)$ in
$L^p(\Omega,\nu)$, for $1<p<\infty$.

For this latter, differently from the Ornstein--Uhlenbeck semigroup in the whole space, 
no explicit integral representation which allows for direct computations is known.
In this direction, in \cite{LMP} the authors consider the restrictions to an open convex set $\Om\subseteq X$ of
$BV(X,\gamma)$ functions and they characterise the finiteness of their total variation in $\Om$ in terms of the
Neumann Ornstein--Uhlenbeck semigroup defined in $\Om$.

Following the ideas in \cite{AM}, we define the $BV(\Omega,\nu)$ space through an integration by parts formula against
suitable Lipschitz functions. Then we show that the functions $u$ of bounded variation in $\Om$ with respect to $\nu$ can be
characterised by the finiteness of the limit of
$\|D_HT_\Omega(t)u\|_{L^1(\Omega,\nu;H)}$ as $t\to 0^+$. The proof of this result relies on a commutation formula between the semigroup $T_\Om(t)$ and the gradient operator along $H$ (see Proposition \ref{prop-comm}). This result was already known in the case of the whole space (see \cite{DPL18}). Here, by means of the crucial pointwise gradient estimate \eqref{gra-est1} and suitable penalisations $\Phi_\varepsilon$ of $U$ outside $\Omega$ based on the {\em distance function from $\Omega$ along $H$} (here is a first point where the convexity of $\Omega$ comes into the play) and the penalisation $\nu_\varepsilon:=e^{-\Phi_\varepsilon}\gamma$ of the measure $\nu$, see Subsection \ref{subsect_sem}, we are able to let $\varepsilon$ to $0^+$ and to to come back to $\Omega$.

Finally we provide a necessary condition in order that a set $E$ is of finite perimeter in $\Om$ with respect to $\nu$ (i.e.,
$\chi_E\in BV(\Omega,\nu)$). This condition is given in terms of the short-time behaviour of the 
{\em Ornstein--Uhlenbeck content} $\|T_\Omega(t)\chi_E-\chi_E\|_{L^1(\Omega,\nu)}$ as $t\to 0^+$. 
Further, a sufficient condition in terms of a related quantity is also shown. and a sufficient condition in terms of a related quantity.
This circle of ideas goes back to \cite{Led94Sem}, which originated several researches. Among these,
the only infinite dimensional result, proved for $BV$ functions in space endowed with a Gaussian measure, is
in \cite{AMP15AGMS}.

\makeatletter \@addtoreset{equation}{subsection}
\def\theequation{\thesection.\arabic{equation}}
\makeatother \makeatletter

\section{Hypotheses and preliminaries}\label{prel}

Let $H_1$ and $H_2$ be two real Hilbert spaces with inner products $\gen{\cdot,\cdot}_{H_1}$ and $\gen{\cdot,\cdot}_{H_2}$
respectively. We denote by $\mathcal{B}(H_1)$ the $\sigma$-algebra of Borel subsets of $H_1$ and by
$C^k_b(H_1;H_2)$, $k\in \N\cup\{\infty\}$ the set of $k$-times Fr\'echet differentiable functions from $H_1$ to $H_2$ with bounded derivatives up to order $k$ ($C_b^k(H_1)$ if $H_2=\R$). For $\Phi\in C_b^1(H_1;H_2)$ we denote by $\mathcal{D} \Phi(x)$ the derivative of $\Phi$ at $x\in H_1$: if $f\in C_b^1(H_1)$, for every $x\in H_1$ there exists a unique $k\in H_1$ such that
$\mathcal{D} f(x)(h)=\gen{h,k}_{H_1}$, $h\in H_1$ and we set $D f(x):=k$.
Let $X$ be a separable Hilbert space, with inner product $\langle\cdot,\cdot\rangle$ and norm $|\cdot|$.
Let $B\in\mathcal{L}(X)$ (the set of bounded linear operators from $X$ to itself). We say that $B$ is \emph{non-negative}   
if $\gen{Bx,x}\geq 0$ for every $x\in X$ and \emph{positive}
if $\gen{Bx,x}>0$ for every $x\in X\setminus\set{0}$. We recall that a non-negative and self-adjoint operator 
$B\in\mathcal{L}(X)$ is a trace class operator whenever ${\rm Tr}(B):=\sum_{n=1}^{\infty}\langle Be_n,e_n\rangle<\infty$
for some (and hence, every) orthonormal basis $(e_n)_{n\in\N}$ of $X$.

Let $\gamma$ be a nondegenerate Gaussian measure on $X$ with mean zero and covariance operator $Q_\infty:=-QA^{-1}$, where the operators $Q$ and $A$ satisfy the following assumptions.

\begin{hyp}\label{hyp_base}
\begin{enumerate}[\rm(i)]
\item $Q\in \mathcal{L}(X)$ is a self-adjoint and non-negative operator with ${\rm Ker}\, Q=\{0\}$;\label{hyp_base_1}
\item $A:D(A)\subseteq X\ra X$ is a self-adjoint operator satisfying $\gen{Ax,x}\leq -\omega\abs{x}^2$ for every $x\in D(A)$ and some positive $\omega$;
\item $Qe^{tA}=e^{tA}Q$ for any $t\geq 0$;\label{hyp_base_3}
\item ${\rm Tr}(-QA^{-1})<\infty$.\label{hyp_base_4}
\end{enumerate}
\end{hyp}
\noindent
Under Hypotheses \ref{hyp_base}\eqref{hyp_base_1}-\eqref{hyp_base_3}, the measure $\gamma$ is well defined and the Ornstein--Uhlenbeck semigroup defined via the Mehler formula
\begin{align}\label{Mehler}
(S(t)f)(x):=\int_Xf(e^{-t}x+\sqrt{1-e^{-2t}}y)d\gamma(y),\qquad x\in X,\,f\in L^1(X,\gamma),
\end{align}
is symmetric in $L^2(X,\gamma)$. We fix an orthonormal basis $(v_k)_{k\in\N}$ of $X$ such that
\begin{equation}\label{qinfty}
Q_\infty v_k=\lambda_k v_k,\qquad k\in\N,
\end{equation}
where $(\lambda_k)_{k\in\N}$ is the decreasing sequence of eigenvalues of $Q_\infty$. Under Hypothesis \ref{hyp_base}\eqref{hyp_base_4}, the Cameron--Martin $(H, \abs{\cdot}_H)$
\[
H:=Q_\infty^{1/2}(X)=\Bigl\{x\in X\Big| \sum_{k=1}^{\infty}\lambda_k^{-1}\gen{x,v_k}^2<\infty\Bigr\},
\]
where $\abs{\cdot}_H$ is induced by the inner product
$\gen{h,k}_H:=\langle Q_\infty ^{-1/2}h,Q_\infty ^{-1/2}k\rangle$,
is a Hilbert space compactly and densely embedded in $X$ (see \cite{Bog} and \cite{DA-ZA1} for further details).
The sequence $(e_k)_{k\in \N}$, where
$e_k=\sqrt{\lambda_k}v_k$ for any $k\in \N$, is an orthonormal basis of $H$.
By Hypotheses \ref{hyp_base}, the operator $-Q^{-1}_\infty:D(Q^{-1}_\infty)\subseteq X\ra X$
($-Q^{-1}_\infty:D(Q^{-1}_\infty)\subseteq H\ra H$, respectively) is the generator of a contractive and strongly continuous
semigroup $e^{-tQ_\infty^{-1}}$ on $X$ (on $H$, respectively), see \cite[Proposition p. 84]{EN1}).
If $Y$ is a Banach space with norm $\norm{\cdot}_Y$, a function $F:X\ra Y$ is said to be $H$-Lipschitz continuous if
there exists a positive constant $C$ such that
\begin{gather}\label{costante di H-lip}
\norm{F(x+h)-F(x)}_Y\leq C\abs{h}_H,
\end{gather}
for every $h\in H$ and $\gamma$-a.e. $x\in X$. We denote by $[F]_{H\text{-Lip}}$ the best constant $C$ in
\eqref{costante di H-lip}. For more information see \cite[Sections 4.5 and 5.11]{Bog}.
We denote by $\mathcal{H}_2$ the space of the Hilbert--Schmidt operators in $H$, that is the space of the bounded linear operators $B:H\ra H$ such that $\norm{B}_{\mathcal{H}_2}^2:=\sum_{i=1}^{\infty}\abs{Bg_i}^2_H$
is finite, where  $\{g_n\,|\,n\in\N\}$ is any orthonormal basis of $H$. We say that
$f:X\rightarrow\R$ is $H$-differentiable at $x_0\in X$ if there exists $\ell\in H$ such that
\[
f(x_0+h)=f(x_0)+\gen{\ell,h}_H+o(|h|_H),\qquad\text{as $|h|_H\rightarrow0$.}
\]
In such a case we set $D_H f(x_0):=\ell$ and $D_i f(x_0):= \langle D_H f(x_0), e_i\rangle_H$ for any $i\in \N$.
The derivative $D_H f(x_0)$ is called the \emph{Malliavin derivative} of $f$ at $x_0$. In a similar way we say that $f$
is twice $H$-differentiable at $x_0$  if $f$  is $H$-differentiable near $x_0$ and there exists $\mathcal B\in \mathcal{H}_2$
such that
\[
f(x_0+h)=f(x_0)+\gen{D_H f(x_0),h}_H+\frac{1}{2}\langle \mathcal B h,h\rangle_H+o(|h|^2_H),\qquad
\text{as $|h|_H\rightarrow0$.}
\]
In such a case we set $D^2_H f(x_0):=\mathcal B$ and $D_{ij} f(x_0):= \langle D^2_H f(x_0)e_j, e_i\rangle_H$ for any
$i,j\in \N$. If $f$ is twice $H$-differentiable at $x_0$, then $D_{ij}f(x_0)=D_{ji}f(x_0)$ for every
$i,j\in\N$. Notice that if $f:X\rightarrow\R$ is once or twice Fr\'echet differentiable at $x_0$ then it is 
once or twice $H$-differentiable at $x_0$ and it holds $D_Hf(x_0)=Q_{\infty}Df(x_0)$,
and $D^2_Hf(x_0)=Q_{\infty}D^2 f(x_0)Q_{\infty}$, where the equality must be understood as holding in $H$.
For any $k\in\N\cup\set{\infty}$, we denote by $\mathcal{F}C_b^k(X)$, the space of cylindrical $C^k_b$ functions, i.e.,
the set of functions $f:X\to \R$ such that $f(x)=\varphi(\langle x, h_1\rangle,\ldots, \langle x, h_N\rangle)$ for
some $\varphi \in C_b^k(\R^N)$, $h_1,\ldots, h_N\in H$ and $N\in\N$.
By $\mathcal{F}C_b^k(X,H)$ we denote $H$-valued cylindrical $C_b^k$ functions with finite rank.
The Sobolev spaces in the sense of Malliavin $D^{1,p}(X,\gamma)$ and $D^{2,p}(X,\gamma)$ with $p\in[1,\infty)$, are defined as the completions of the \emph{smooth cylindrical functions} $\mathcal{F}C_b^\infty(X)$ in the norms
\begin{gather*}
\norm{f}_{D^{1,p}(X,\gamma)}:=
\Bigl(\norm{f}^p_{L^p(X,\gamma)}+\int_X\abs{D_H f}_H^pd\gamma\Bigr)^{\frac{1}{p}};
\\
\norm{f}_{D^{2,p}(X,\gamma)}:=
\Bigl(\norm{f}^p_{D^{1,p}(X,\gamma)}+\int_X\|D_H^2 f\|^p_{\mathcal{H}_2}d\gamma\Bigr)^{\frac{1}{p}}.
\end{gather*}
This is equivalent to considering the domain of the closure of the gradient operator, defined on smooth cylindrical functions, in $L^p(X,\gamma)$ (see \cite[Section 5.2]{Bog}). Let $U:X\ra\R$ satisfy the following assumptions.
\begin{hyp}\label{ipotesi peso}
$U$ is a convex function which belongs to $C^2(X)\cap D^{1,q}(X,\gamma)$ for all $q\in[1,\infty)$ with
$H$-Lipschitz gradient.
\end{hyp}
The convexity of the function $U$ guarantees that $U$ is bounded from below by a linear function, therefore it decreases at most linearly and by Fernique theorem (see \cite[Theorem 2.8.5]{Bog}) $e^{-U}$ belongs to $L^1(X,\gamma)$. Then we can consider the finite log-concave measure
\begin{equation*}
\nu:= e^{-U}\gamma.
\end{equation*}
It is obvious that $\gamma$ and $\nu$ are equivalent measures, hence saying that a statement holds $\gamma$-a.e. is the same as saying that it holds $\nu$-a.e.
Moreover as $U\in \cap_{q\geq 1}D^{1,q}(X,\gamma)$, the operator $D_H:\mathcal{F}C^1_b(X)\ra L^p(X,\nu;H)$ is closable
in $L^p(X,\nu)$, $p\in(1,\infty)$ and the space $D^{1,p}(X,\nu)$, $p>1$ can be defined as the domain of its closure
(still denoted by $D_H$). In a similar way we may define $D^{2,p}(X,\nu)$, $p\in(1,\infty)$ (for more details see
\cite{AD-CA-FE1,CF16,Fer15}). The Gaussian integration by parts formula
$\int_X D_i f d\gamma=\frac{1}{\sqrt{\lambda_i}}\int_X \langle x,v_i\rangle f d\gamma$, which holds true for any
$f\in \mathcal{F}C^1_b(X)$ and $i\in\N$, yields
\begin{gather*}
\int_X \psi D_i \varphi d\nu+\int_X\varphi D_i\psi d\nu=\int_X \varphi\psi D_i U d\nu+\frac{1}{\sqrt{\lambda_i}}
\int_X \langle x, v_i\rangle\varphi\psi d\nu,\qquad\;\,i \in \N,
\end{gather*}
for any $\varphi\in D^{1,p}(X,\nu)$ ($p>1$) and $\psi \in \mathcal{F}C^1_b(X)$.

In what follows $\Omega$ denotes an open subset of $X$. In this case, the spaces $D^{1,p}(\Omega,\nu)$ and
$D^{2,p}(\Omega,\nu)$, $p\in(1,\infty)$, can be defined in a similar way as in the whole space, thanks to the following
result (see \cite[Proposition 1.4]{AFP18}).

\begin{proposition}
Assume that Hypotheses $\ref{hyp_base}$ and $\ref{ipotesi peso}$ are satisfied. Let $p\in(1,\infty)$ and let
$\Om$ be an open subset of $X$. The operators $D_H:\mathcal{F}C_b^\infty(\Omega)\ra L^p(\Omega,\nu; H)$ and
\begin{gather*}
(D_H,D_H^2):\mathcal{F}C_b^\infty(\Omega)\times \mathcal{F}C_b^\infty(\Omega)\ra
L^p(\Omega,\nu; H)\times L^p(\Omega,\nu; \mathcal{H}_2)
\end{gather*}
are closable in $L^p(\Omega,\nu)$ and $L^p(\Omega,\nu)\times L^p(\Omega,\nu)$, respectively. Here $\mathcal{F}C_b^\infty(\Omega)$ is the space of the restrictions to $\Omega$ of functions in $\mathcal{F}C_b^\infty(X)$.
\end{proposition}

\noindent The spaces $D^{1,p}(\Omega,\nu;H)$, $p\in(1,\infty)$, are defined in a similar way, replacing smooth cylindrical functions with $H$-valued smooth cylindrical functions with finite rank. We recall that if $F\in D^{1,p}(\Omega,\nu;H)$, then $D_H F(x)$ belongs to $\mathcal{H}_2$ for a.e. $x\in \Om$. We denote by $p'$ the conjugate exponent to $p\in (1,\infty)$.

\subsection{Perturbed Ornstein--Uhlenbeck semigroup on convex domains}\label{subsect_sem}

In order to consider the initial boundary value problems defined in $\Om$ we define
the \emph{distance function along $H$}
\begin{gather*}
d_\Om(x):=\left\{\begin{array}{lc}
\inf\{|h|_H\,|\, h\in H\cap (\Omega-x)\}, \quad \quad& H\cap (\Omega-x)\neq\emptyset;\\
\infty, & H\cap (\Omega-x)=\emptyset,
\end{array}\right.
\end{gather*}
$x\in X$, and we recall some useful regularity results, (see, for instance, \cite[Theorems 2.8.5 and 5.11.2]{Bog} and
\cite[Section 3]{CF18}).

\begin{proposition}\label{prop_dist}
If $\Omega\subseteq X$ be an open convex set, then $d_\Omega^2$ is $H$-differentiable and its Malliavin derivative is $H$-Lipschitz with $H$-Lipschitz constant less than or equal to $2$, i.e.,
\[|D_Hd_\Om^2(x+h)-D_Hd_{\Om}^2(x)|_H\le 2 |h|_H,\]
for any $h\in H$ and for $\nu$-a.e $x\in X$. Moreover $D_H^2d_\Omega^2$ exists $\nu$-a.e. in $X$ and $d_\Omega^2$ belongs to $D^{2,p}(X,\nu)$ for every $p\in[1,\infty)$.
\end{proposition}

We require some further regularity on $d_\Omega^2$.

\begin{hyp}\label{ipo_convex}
Let $\Om$ be an open convex subset of $X$ such that $\nu(\partial\Omega)=0$ and $D_H^2d_\Omega^2$ is $H$-continuous $\gamma$-a.e. in $X$, i.e., for $\gamma$-a.e. $x\in X$ we have 
\[
\lim_{\abs{h}_H\ra 0}D^2_H d^2_\Omega(x+h)=D^2_H d^2_\Omega(x).
\]
\end{hyp}

\begin{remark}{\rm 
As stated in \cite[Remark 1.7]{AFP18} there is a rather large class of subsets of $X$ satisfying Hypothesis \ref{ipo_convex}. For instance if $\partial \Omega$ is (locally) a $C^2$-embedding in $X$ of an open subset of a hyperplane in $X$ and $\nu(\partial\Omega)=0$, then Hypothesis \ref{ipo_convex} is satisfied. Easy examples are open balls and open ellipsoids of $X$, open hyperplanes of $X$ and every set of the form $\Omega=\set{x\in X\tc G(x)< 0}$, where $G:X\ra\R$ is a $C^2$-convex function such that $D_H G$ is non-zero at every point of $\partial \Omega$.
}\end{remark}

\noindent
We consider the semigroup $T_\Om(t)$ on $L^2(\Om,\nu)$ and its generator $L_\Omega$:
\begin{align}\label{defn_LOm}
D(L_{\Omega})=&\Bigl\{u\in D^{1,2}(\Omega,\nu)\,\Big|\, \exists v\in L^2(\Omega,\nu)\text{ such that }
\\
&\int_\Omega\langle D_Hu,D_H\varphi\rangle_Hd\nu
=-\int_\Omega v\varphi\, d\nu\ \forall\varphi\in\mathcal{F}C^\infty_b(\Om)\Bigr\}
\notag\end{align}
with $L_{\Omega}u:=v$ if $u\in D(L_{\Omega})$.
We recall (see \cite[Section 2]{AFP18}) an approximation procedure of $T_\Om(t)f$, when $f\in L^2(\Om, \nu)$,
through $\mathcal{F}C^3_b(X)$ functions that relies on reduction to a finite (say $n$-) dimensional space and on
a $\varepsilon$-penalisation argument. Accordingly, the approximation depends on two parameters $n$ and $\varepsilon$. More precisely, we consider the function $\Phi_\eps:X\to \R$ defined by
\[
\Phi_\eps(x):=U(x)+\frac{1}{2\eps}d^2_\Omega(x),\qquad\;\, x\in X,\, \eps>0,
\]
and the measure $\nu_\ve$ given by $e^{-\Phi_\eps}\gamma$. Next, we consider the operator $L_\ve$ on the whole $X$ defined as
\begin{align}\label{operator_veps}
D(L_{\eps})=&\Bigl\{u\in D^{1,2}(X,\nu_\eps)\,\Big|\, \exists\ v\in L^2(X,\nu_\eps)\text{ such that }
\\
&\int_X\gen{D_H u,D_H \varphi}_Hd\nu_\eps=-\int_X v\varphi\, d\nu_\eps
\text{ for every }\varphi\in\mathcal{F}C^\infty_b(X)\Bigr\},\notag
\end{align}
with $L_{\eps}u:=v$ if $u\in D(L_{\eps})$, and
the semigroup $T_\eps(t)$ generated by $L_\eps$ in $L^2(X, \nu_\eps)$. We point out that $L_\eps$ acts on smooth cylindrical functions $\varphi$ as follows
\begin{align*}
L_\eps\varphi&={\rm Tr}(D^2_H \varphi)-\sum_{i=1}^{\infty}\lambda_i^{-1}\langle x,e_i\rangle D_i\varphi-\langle D_H\Phi_\eps,D_H\varphi\rangle_H\\
&={\rm Tr}(D^2_H \varphi)-\sum_{i=1}^{\infty}\lambda_i^{-1}\langle x,e_i\rangle D_i\varphi
-\Big\langle D_HU+\frac{1}{2\eps}D_H d^2_\Om,D_H\varphi\Big\rangle_H.
\end{align*}
Now we recall a useful approximation result whose proof can be found in \cite[Theorem 2.8]{AFP18}.

\begin{theorem}\label{approx theorem}
Under Hypotheses $\ref{hyp_base}$, $\ref{ipotesi peso}$ and $\ref{ipo_convex}$ the following statements hold true.
\begin{enumerate}[\rm(i)]
\item For any $\eps>0$ and $f\in L^2(X,\nu_\eps)$, there exists a sequence $(f_n)_{n\in\N}\subseteq L^2(X,\nu_\eps)$ converging to $f$ in $L^2(X, \nu_\eps)$ such that $T_\eps(t)f_n$ is in $\mathcal{F}C^3_b(X)$ and
\begin{equation*}
\lim_{n\to\infty}\norm{T_{\eps}(t)f_n-T_\eps(t) f}_{D^{1,2}(X,\nu_\eps)}=0,\qquad\;\, t>0.
\end{equation*}
In addition, if $f\in D^{1,2}(X,\nu_\ve)$  then the sequence $(f_n)$ can be chosen in a way that $D_Hf_n$ converges to $D_Hf$ in $L^1(X, \nu_\ve; H)$, as $n\to\infty$.
\item For any $f\in L^2(\Omega,\nu)$ there exists an infinitesimal sequence $(\eps_n)_{n\in\N}$ such that
$T_{\eps_n}(t) \tilde{f}$ weakly converges to $T_\Omega(t) f$ in $D^{1,2}(\Omega,\nu)$, where $\tilde{f}$ is any $L^2$-extension of $f$ to $X$.
\end{enumerate}
\end{theorem}
We collect some properties of $T_\Om(t)$, see \cite[Proposition 1.10, Theorems 3.1 \& 3.3]{AFP18}.

\begin{proposition}\label{main_properties}
If Hypotheses $\ref{hyp_base}$, $\ref{ipotesi peso}$ and $\ref{ipo_convex}$ hold true, then
\begin{enumerate}[\rm(i)]
\item the semigroup $T_\Omega(t)$ generated in $L^2(\Om,\nu)$ can be extended to a positivity preserving contraction
semigroup in $L^p(\Omega,\nu)$ for every $1\leq p\leq \infty$ and $t\geq 0$, still denoted by $T_\Omega(t)$. It
is strongly continuous in $L^p(\Om, \nu)$ for any $p \in [1,\infty)$ and consistent;
\item for any $p\in[1,\infty)$, $f\in L^p(\Omega,\nu)$ and $g\in L^\infty(\Omega,\nu)$ it holds
\begin{equation}\label{invariance}
\int_\Omega f T_\Omega(t)gd\nu=\int_\Omega gT_\Omega(t)fd\nu,\qquad t>0;
\end{equation}
\item for any $p\in(1,\infty)$, $f\in L^p(\Om,\nu)$ and $t>0$ there is $K_p>0$ such that
\begin{gather}\label{gra-est1}
|D_HT_\Om(t)f|_H^p\leq K_pt^{-p/2}T_\Om(t)|f|^p\qquad \nu\text{-a.e. in }\Om;
\end{gather}
\item if $f\in D^{1,p}(\Om,\nu)$, $t>0$ and $p\in[1,\infty)$ it holds
\begin{equation}\label{gra-est}
|D_H T_\Om(t) f|^p\leq e^{-p\lambda_1^{-1}t}T_\Om(t)|D_H f|^p_H\qquad \nu\text{-a.e. in }\Om.
\end{equation}
\end{enumerate}
\end{proposition}
We point out that the results in Proposition \ref{main_properties} continue to hold if we replace $\Om$, $\nu$ and
$T_\Om(t)$ by $X$, $\nu_\eps$ and $T_\eps(t)$, respectively.

\subsection{BV functions in Hilbert spaces: definitions and some known facts}

We introduce $BV$ functions in the Wiener space setting. Let $Y$ be a separable Hilbert space with norm 
$\abs{\cdot}_Y$. We recall that in separable
spaces the $\sigma$-algebra $\mathcal{B}(X)$ is generated by the family of the cylindrical sets (see e.g. \cite{VTC}).
Denote by $\mathcal{M}(\Om;Y)$ the set of Borel
$Y$-valued measures on $\Om$. If $Y=\R$ then we write $\mathcal{M}(\Om)$. The total variation of
$\mu\in\mathcal{M}(\Om;Y)$ is the positive Borel measure
\begin{align*}
|\mu|(B):=\sup\set{\sum_{n=1}^{\infty}|\mu(B_n)|_Y\tc
\begin{array}{c}
 B=\bigcup_{n=1}^{\infty}B_n,\ B_n \in \mathcal{B}(\Om)\\
 B_n\cap B_m=\emptyset, \text{ if } n\neq m,
 \end{array}},\ \ B\in \mathcal{B}(\Omega).
\end{align*}
Let ${\rm Lip}_c(\Omega;Y)$ be the set of bounded Lipschitz continuous $Y$-valued functions $g:\Omega\ra Y$ such that
${\rm dist}({\rm supp}\, g,X\smallsetminus \Omega)>0$ and define the space $BV(\Om,\nu)$ as follows.

\begin{definition}
Let $\Omega$ be an open subset of $X$. We say that a function $f\in L^2(\Omega,\nu)$ is of bounded variation in $\Omega$, and we write $f\in BV(\Omega,\nu)$, if there exists a measure $\mu\in\mathcal{M}(\Omega;H)$ such that
\begin{gather*}
\int_\Omega f\partial_h^* gd\nu=-\int_\Omega gd\gen{\mu,h}_H,
\end{gather*}
for every $g\in {\rm Lip}_c(\Omega)$ and $h\in H$, where $\partial_h^*$ denotes, up to the sign,
the adjoint in $L^2(\Om,\nu)$ of the partial derivative along $h\in H$. In this case we set $D_\nu f:=\mu$.
\end{definition}
As in the finite dimensional case, one can characterise functions of bounded variation by
their total variation.

\begin{definition}
Let $\Omega$ be an open subset of $X$ and $u\in L^2(\Omega,\nu)$. We define the variation of $u$ in $\Omega$ by
\begin{gather*}
V_\nu(u,\Omega):=\sup\set{\int_\Omega u\,{\rm div}_\nu\, gd\nu\tc \begin{array}{c}
F\subseteq H\text{ \rm  finite dimensional},\\
g\in{\rm Lip}_c(\Omega;F),\ \norm{g}_\infty\leq 1.
\end{array}}.
\end{gather*}
Here ${\rm div}_\nu\, g= \sum_{i=1}^N \partial_{k_i}^*g_i(x)$ if $g(x)=\sum_{i=1}^N g_i(x)k_i$ and 
$F={\rm span}\{k_1,\ldots, k_N\}$ for some $N\in\N$.
\end{definition}
\noindent When $\Om=X$, in the two definitions above we can consider ${\rm Lip}_b(X)$ and ${\rm Lip}_b(X;F)$ respectively, as test functions spaces.

As announced, in \cite[Theorem 5.7]{AM} it has been proved that
$u\in BV(\Omega,\nu)$ if and only if $V_\nu(u,\Omega)$ is finite. Moreover, in this case
\begin{gather}\label{mis=var}
|D_\nu u|(\Omega)=V_\nu(u,\Omega).
\end{gather}
Finally we say that a Borel subset $E$ of $X$ is of finite perimeter in $\Omega$ with respect to $\nu$, whenever the function $\chi_E$ belongs to $BV(\Om,\nu)$. In this case we denote by $P_\nu(E,\Omega)$ the total variation of $\chi_E$ in $\Om$.

\section{A De Giorgi type characterisation}

The main result of this section is the De Giorgi type characterisation of $BV(\Om, \nu)$ functions in
Theorem \ref{bv_dg}, which relies on a ``quasi-commutative'' formula between the semigroup $T_\Omega(t)$
and the $H$-gradient operator $D_H$; here estimate \eqref{gra-est} plays a crucial role.
This formula is inspired by an analogous formula proved in \cite{DPL18}.
We first define the Sobolev spaces $D^{1,2}(X,\nu_\eps;H)$.

\begin{definition}
We denote by $D^{1,2}(X,\nu_\eps;H)$ the domain of the closure of the operator $D_H:\mathcal{F}C_b^1(X,H)\ra L^2(X,\nu_\eps;\mathcal{H}_2)$ in the $L^2(X,\nu_\eps;H)$ norm $($see \cite[Section 8.1]{Bog2}$)$. $D_H$ is defined as
\[
D_H\Phi(x)=\sum_{i=1}^n\sum_{j=1}^{k(i)}\frac{\partial \varphi_i}{\partial \xi_j}(\langle x,x_1\rangle,\ldots,\langle x,x_{k(i)}\rangle)((Q^{1/2}_\infty x_j)\otimes e_i),
\]
where $\{e_i\,|\,i\in\N\}$ is an orthonormal basis of $H$ and
\[
\Phi(x)=\sum_{i=1}^n\varphi_i(\langle x,x_1\rangle,\ldots,\langle x,x_{k(i)}\rangle) e_i
\]
for some $n\in\N$, $k(i)\in\N$, $x_1,\ldots, x_{k(i)}\in X$ and
$\varphi_i\in C_b^1(\R^{k(i)})$ for every $i=1,\ldots,n$. In an analogous way we define the space $D^{1,2}(\Omega,\nu;H)$.
\end{definition}

We first show a vector-valued version of Theorem \ref{approx theorem}. Let ${\bf L}_{\eps}$ in $L^2(X,\nu_\eps;H)$
be the operator defined via the quadratic form by
\[
(F,G)\mapsto\int_X\gen{D_H F,D_H G}_{\mathcal{H}_{\,2}}d\nu_\eps\qquad F,G\in D^{1,2}(X,\nu_\eps;H).
\]
In the same way we define the operator ${\bf L}_\Omega$ in in $L^2(\Om,\nu;H)$.
We recall that by \cite[p. 84]{EN1} (\cite[Corollary 3.17 and Proposition 3.23]{EN1} and \cite[Corollary 4.8]{EN1}, respectively), the operators ${\bf L}_{\eps}$ and ${\bf L}_{\Omega}$ generate strongly continuous semigoups
${\bf T}_\eps(t)$ and ${\bf T}_\Omega(t)$ (contractive and analytic, respectively).

\begin{proposition}\label{component}
The operators ${\bf L}_\eps$, ${\bf L}_\Omega$ and the semigroups ${\bf T}_\eps(t)$ and
${\bf T}_\Omega(t)$ act component by component, i.e., if $F\in D({\bf L}_\eps)$ $(D({\bf L}_\Omega)$, respectively$)$, and 
it is such that $F=\sum_{i=1}^{\infty}f_ie_i$ for some basis $\{e_n\,|\,n\in\N\}$ of $H$, then 
$f_i\in D(L_\eps)$ $(D(L_\Omega)$, respectively$)$ and
\[
{\bf L}_\eps F=\sum_{i=1}^{\infty}(L_\eps f_i)e_i,\qquad {\bf L}_\Omega F=
\sum_{i=1}^{\infty}(L_\Omega f_i)e_i .
\]
Moreover for every $t>0$, if $F\in L^2(X,\nu_\eps;H)$ $(L^2(\Omega,\nu;H)$, respectively$)$, and it is such that
$F=\sum_{i=1}^{\infty}f_ie_i$ for some basis $\{e_n\,|\,n\in\N\}$ of $H$ and $f_i\in L^2(X,\nu_\eps)$
($L^2(\Omega,\nu)$, respectively) then
\[
{\bf T}_\eps(t)F=\sum_{i=1}^{\infty}(T_\eps(t)f_i)e_i,\qquad \pa{{\bf T}_\Omega(t)F=
\sum_{i=1}^{\infty}(T_\Omega(t)f_i)e_i\text{, respectively}}.
\]
The above identities hold $\nu_\eps$-a.e. in $X$ ($\nu_\Om$-a.e. in $\Om$, respectively).
\end{proposition}

\begin{proof}
We only show the results for ${\bf L}_\eps$ and ${\bf T}_\eps(t)$.
Let $F=\sum_{i=1}^{\infty}f_ie_i \in D({\bf L}_\eps)$ and let $G=ge_j$ for some $j\in\N$ and $g\in D^{1,2}(X,\nu_\eps)$; then
\begin{align*}
\int_X\gen{D_H g,D_H f_j}_Hd\nu_\eps &=\int_X\gen{D_H G,D_H F}_{\mathcal{H}_2}d\nu_\eps
\\
&=-\int_X\gen{G,{\bf L}_\eps F}_Hd\nu_\eps=-\int_X g({\bf L}_\eps F)_jd\nu_\eps.
\end{align*}
This shows that $f_j\in D(L_\eps)$ (see \eqref{operator_veps}) and $L_\eps f_j=({\bf L}_\eps F)_j$.
Now observe that
\[
D_t({\bf T}_\eps(t)F)_j=({\bf L}_\eps {\bf T}_\eps(t)F)_j=L_\eps({\bf T}_\eps(t)F)_j,\qquad ({\bf T}_\eps(0)F)_j=f_j.
\]
Thus, by the uniqueness of the solution of the Cauchy problem associated with $D_t-L_\eps$ in $L^2(X, \nu_\eps)$,
it follows that $({\bf T}_\eps(t)F)_j=T_\eps(t)f_j$ for any $t>0$. The arbitrariness of $j\in\N$ concludes the proof.
\end{proof}

\begin{remark}
According to the definition of ${\bf T}_\eps(t)$ and ${\bf T}_\Omega(t)$ it is immediately seen that for every
$F\in L^2(X, \nu_\eps;H)$ and $G\in L^2(\Om,\nu;H)$
\begin{equation}\label{comparison}
|{\bf T}_\eps(t)F|^2\le T_\eps(t)|F|^2, \qquad t\ge 0,\  \nu_\eps\text{\rm -a.e. in } X
\end{equation}
and
\begin{equation}\label{comparison1}
|{\bf T}_\Om(t)G|^2\le T_\Om(t)|G|^2, \qquad t\ge 0,\  \nu\text{\rm -a.e. in } \Om.
\end{equation}
\end{remark}
Moreover, taking into account that the semigroups ${\bf T}_\Omega(t)$ and ${\bf T}_\eps(t)$ act
component by component, we can obtain a vector-valued version of Theorem \ref{approx theorem}.

\begin{theorem}\label{Appr-vect}
Under Hypotheses $\ref{hyp_base}$, $\ref{ipotesi peso}$ and $\ref{ipo_convex}$, the following statements hold true.
\begin{enumerate}[\rm(i)]
\item For any $\eps>0$ and $F\in L^2(X,\nu_\eps;H)$, there exists a sequence $(F_n)_{n}\subseteq L^2(X,\nu_\eps;H)$ such that ${\bf T}_\eps(t)F_n$ belongs to $\mathcal{F}C^3_b(X;H)$ and
\begin{align}
&\lim_{n\ra\infty}\norm{F_n-F}_{L^{2}(X,\nu_\eps;H)}=0,\notag \\
&\lim_{n\ra \infty}\norm{{\bf T}_{\eps}(t)F_n-{\bf T}_\eps(t) F}_{D^{1,2}(X,\nu_\eps;H)}=0,\qquad\;\, t>0.\label{radiosveglia1}
\end{align}
If, in addition, $F\in D^{1,2}(X,\nu_\ve;H)$  then $D_H F_n$ converges to $D_H F$ in $L^1(X, \nu_\ve; \mathcal{H}_2)$, as $n\to \infty$.
\item For any $F\in L^2(\Omega,\nu;H)$ there exists an infinitesimal sequence $(\eps_n)_{n\in\N}$ such that ${\bf T}_{\eps_n}(t) \tilde{F}$ weakly converges to ${\bf T}_\Omega(t) F$ in $D^{1,2}(\Omega,\nu;H)$, where $\tilde{F}$ is any $L^2$-extension of $F$ to $X$.
\end{enumerate}
\end{theorem}

\begin{proof}
(i) Let $F=\sum_{i=1}^{\infty}f^{(i)}e_i$ where $f^{(i)}\in L^2(X,\nu_\eps)$, $i\in\N$. For every $i\in\N$,
by Theorem \ref{approx theorem}(i), there exists $(f_k^{(i)})_{k\in\N}\subseteq L^2(X,\nu_\eps)$ converging to $f^{(i)}$
in $L^2(X,\nu_\eps)$ such that $T_\eps(t)f_k^{(i)}$ belongs to $\mathcal{F}C^3_b(X)$ and
\begin{align}\label{radiosveglia2}
\lim_{k\ra \infty}\|T_\eps(t)f_k^{(i)}-T_\eps(t)f^{(i)}\|_{D^{1,2}(X,\nu_\eps)}=0,\qquad t, \eps>0.
\end{align}
Observe that \eqref{radiosveglia1} follows immediately from \eqref{radiosveglia2}. Now fix $i,n\in\N$ and consider $k_i\in\N$ such that for every $k\geq k_i$ it holds
\[
\int_X|f_k^{(i)}-f^{(i)}|^2d\nu_\eps<\frac{1}{n2^i}.
\]
Consider the vector field $F_n:=\sum_{i=1}^n f_{k_i}^{(i)}e_i$. We claim that $(F_n)$ is the sequence we are looking for. Indeed $F_n$ belongs to $L^2(X,\nu_\eps;H)$ for any $n\in\N$. Let $n_0\in\N$ be such that $\sum_{i=n_0+1}^{\infty}\|f^{(i)}\|_{L^2(X,\nu_\eps)}^2\leq \eta/2$ and let $n\geq n_0$ such that $1/n<\eta/2$. We have
\[
\norm{F_n-F}_{L^2(X,\nu_\eps;H)}^2 \leq 
\sum_{i=1}^n\int_X|f_{k_i}^{(i)}-f^{(i)}|^2d\nu_\eps+\sum_{i=n+1}^{\infty}\|f^{(i)}\|_{L^2(X,\nu_\eps)}^2
\leq \frac{1}{n}+\frac{\eta}{2}\leq \eta.
\]
In a similar way we can prove the other statements.\\
(ii) is an immediate consequence of Proposition \ref{component} and Theorem \ref{approx theorem}(ii).
\end{proof}

Before going on, recall that usually in the characterisation of functions of bounded variation in terms of the
short-time behaviour of suitable semigroups a crucial tool is an appropriate commutation formula between the semigroup and the gradient operator. For instance, for the Wiener space and the Ornstein--Uhlenbeck semigroup the equality
$D_H S(t)f=e^{-t}\mathbf{S}(t)D_H f$ holds true for any $t \ge 0$.
Let us prove a (quasi) commutation formula between $T_\Om(t)$ and $D_H$, under
the following additional assumption.
\begin{hyp}\label{hyp_d}
The map $(d_\Om)^{-2}\|D_H^2d_\Om^2\|_{{\mathcal H}_{2}}$ belongs to $L^2(X, \nu)$.
\end{hyp}

\begin{remark}
It is not difficult to show that every open ball and every open ellipsoid of $X$ as well as every open hyperplane of $X$ satisfy Hypothesis \ref{hyp_d}. We show that Hypothesis \ref{hyp_d} is satisfied when $\Om$ is the unit ball $B_X$ centered at zero. The other examples can be discussed in a similar fashion. Observe that, by Proposition \ref{prop_dist},
$\|D_H^2d_\Om^2\|_{{\mathcal H}_{2}}\leq 2$ and $\|D_H^2d_\Om^2(x)\|_{{\mathcal H}_{2}}=0$ if $x\in B_X$. Moreover there exists a constant $C>0$ such that
\begin{align*}
d_{B_X}(x)&\geq C\inf\{|h|_X\,|\, h\in H\cap (B_X-x)\}\geq C\inf\{|h|_X\,|\, h\in (B_X-x)\}\\
&= C\inf\{|x-h|_X\,|\, h\in B_X\}=C{\rm dist}(x,B_X)=C||x|_X-1|,
\end{align*}
where ${\rm dist}(x,B_X)$ is the distance of $x$ from $B_X$. So
\begin{align}\label{ball}
&\int_X d_{B_X}^{-4}\|D_H d^2_{B_X}\|^2_{\mathcal{H}_2}d\nu
\leq K\int_{X\smallsetminus B_X}\frac{1}{(|x|_X-1)^4}d\nu(x)
\\
&\leq\! K\!\!\int_{X}\!\frac{1}{(|x|_X-1)^4} e^{-U(x)}d\gamma(x)
\leq K\Bigl(\int_X\! e^{-p'U}d\gamma\Bigr)^{\frac{1}{p'}}\!
\Bigl(\int_{X}\!\frac{1}{(|x|_X-1)^{4p}}d\gamma(x)\Bigr)^{\frac{1}{p}}\notag
\\ \notag
&\leq K\Bigl(\int_X e^{-p'U}d\gamma\Bigr)^{\frac{1}{p'}}
\Bigl(\int_{\R^n}\frac{1}{(|\xi|_{\R^n}-1)^{4p}}d\gamma_n(\xi)\Bigr)^{\frac{1}{p}},
\end{align}
where $K$ is a positive constant and $\gamma_n$ denotes the $n$-dimensional Gaussian measure, image of
$\gamma$ under the projection on ${\rm span}\,\{v_1,\ldots,v_n\}$. To conclude, observe that there exists $n\in\N$
such that the right-hand side of \eqref{ball} is finite.
\end{remark}

\begin{proposition}\label{prop-comm}
Under Hypotheses $\ref{hyp_base}$, $\ref{ipotesi peso}$, $\ref{ipo_convex}$ and $\ref{hyp_d}$, the formula
\begin{align}\label{main}
D_H T_\Omega(t)f &-(e^{-tQ_\infty^{-1}}{\bf T}_{\Omega}(t)D_H f)\notag\\
&\phantom{aaaaaaaa}=-\int_0^te^{(s-t)Q_\infty^{-1}}{\bf T}_\Omega(t-s)(D_H^2 UD_H T_\Omega(s)f)ds.
\end{align}
holds true $\nu$-a.e. in $\Om$, for any $f\in {\rm Lip}_c(\Om)$ and $t>0$.
\end{proposition}

\begin{proof}
In order to prove \eqref{main} we show that
\begin{align}\label{aim}
\int_\Om \langle D_H & T_\Omega(t)f ,G\rangle_H d\nu=
\int_\Om \langle e^{-tQ_\infty^{-1}}{\bf T}_{\Omega}(t)D_H f, G\rangle_H d\nu
\notag\\
&-\int_\Om \int_0^t\langle e^{(s-t)Q_\infty^{-1}}({\bf T}_\Omega(t-s)(D_H^2 UD_H T_\Omega(s)f)), G\rangle_H ds d\nu,
\end{align}
for any $f\in {\rm Lip}_c(\Om)$, $G \in C_b(\Om;H)$ and $t>0$. By performing slight changes in \cite[Appendix A]{DPL18} we get
\begin{equation}\label{comm_eps}
D_H T_\eps(t)g -(e^{-tQ_\infty^{-1}}{\bf T}_{\eps}(t)D_H g)
=-\int_0^te^{(s-t)Q_\infty^{-1}}{\bf T}_\eps(t-s)(D_H^2 \Phi_\eps D_H T_\eps(s)g)ds
\end{equation}
$\nu_\varepsilon$-a.e. in $X$ for any $g\in {\rm {\rm Lip}}_b(X)$ and $\eps>0$, where $T_\eps(t)$ is the semigroup 
introduced in Subsection \ref{subsect_sem}.
Now, let $f\in {\rm Lip}_c(\Om)$ and $\tilde{f}$ be the trivial extension to zero of $f$ in the whole space $X$. Clearly, $\tilde{f}$ belongs to ${\rm {\rm Lip}}_b(X)$ and \eqref{comm_eps} holds true with $g$ replaced by $\tilde{f}$. Consequently, multiplying \eqref{comm_eps} by the function $G$ and integrating on $\Om$ with respect to $\nu$ yield
\begin{align}\label{comm_eps_int}
\int_\Om \langle D_H & T_\eps(t)\tilde{f}, G\rangle_H d\nu=\int_\Om \langle e^{-tQ_\infty^{-1}}({\bf T}_{\eps}(t)D_H \tilde{f}) , G\rangle_H d\nu\notag\\
& -\int_\Om\int_0^t\langle e^{(s-t)Q_\infty^{-1}}({\bf T}_\eps(t-s)(D_H^2 \Phi_\eps D_H T_\eps(s)\tilde{f})), G\rangle_H  ds d\nu\notag\\
&\phantom{aaaaaaaaaaaa}=\int_\Om\langle e^{-tQ_\infty^{-1}}{\bf T}_{\eps}(t)D_H \tilde{f}, G\rangle_H d\nu\notag\\
& -\int_0^t\int_\Om \langle e^{(s-t)Q_\infty^{-1}}({\bf T}_\eps(t-s)(D_H^2 \Phi_\eps D_H T_\eps(s)\tilde{f})), G\rangle_H d\nu ds
\end{align}
where in the last line we used the Fubini--Tonelli theorem.

We split the proof of \eqref{aim} in two steps.

\noindent
{\bf Step 1.} We argue by approximation on the last terms in \eqref{comm_eps_int} and \eqref{aim}.\\
For every $\eps,s>0$ we fix a Borel measurable version of $D_HT_\Om(s)f$ and $D_HT_\eps(s)\tilde{f}$  in $L^2(\Omega,\nu;H)$ and  $L^2(X,\nu_\eps;H)$, respectively. Consider the function
\begin{align*}
\Gamma_\eps(s,x):=\eqsys{D_HT_\Om(s)f(x), &x\in\Om;\\
D_HT_\eps(s)\tilde{f}(x), & x\in X\smallsetminus\Om.}
\end{align*}
Observe that the map $x\mapsto \Gamma_\eps(s,x)$ is an extension of $D_HT_\Om(s)f$ to the whole $X$. Thus,
by Theorem \ref{Appr-vect} there is a sequence $\eps_n\downarrow 0$ such that for every $\eta>0$ the
function ${\bf T}_{\eps_n}(t-s)(D_H^2 U\Gamma_\eta(s,\cdot))$ weakly converges to
${\bf T}_{\Om}(t-s)(D_H^2 UD_HT_\Om(s)f)$ in $D^{1,2}(\Omega,\nu;H)$. Observe that the set
\begin{align}\label{insieme}
\{{\bf T}_{\eps_n}(t-s)(D_H^2 U\Gamma_\eta(s,\cdot))\,|\, n\in\N \text{ and }s,\eta>0\}
\end{align}
is bounded in $L^2(\Omega,\nu;H)$. Indeed by the contractivity of ${\bf T}_{\eps_n}(t)$ in the space
$L^2(X,\nu_{\eps_n};H)$, the fact that $U\equiv \Phi_\eps$ on $\Om$ and estimate \eqref{gra-est} we have
\begin{align*}
&\norm{{\bf T}_{\eps_n}(t-s)(D_H^2 U\Gamma_\eta(s,\cdot))}_{L^2(\Omega,\nu;H)}
=\norm{{\bf T}_{\eps_n}(t-s)(D_H^2 U\Gamma_\eta(s,\cdot))}_{L^2(\Omega,\nu_{\eps_n};H)}
\\
&\leq \norm{{\bf T}_{\eps_n}(t-s)(D_H^2 U\Gamma_\eta(s,\cdot))}_{L^2(X,\nu_{\eps_n};H)}
\\
&\leq\norm{D_H^2 U\Gamma_\eta(s,\cdot)}_{L^2(X,\nu_{\eps_n};H)}
\leq [D_HU]_{H\text{-Lip}}\norm{\Gamma_\eta(s,\cdot)}_{L^2(X,\nu_{\eps_n};H)}
\\
&\leq [D_HU]_{H\text{-Lip}} \left(\|D_HT_{\eta}(s)\tilde{f}\|_{L^2(X,\nu_{\eps_n};H)}
+\norm{D_HT_\Om(s)f}_{L^2(\Om,\nu;H)}\right)
\\
&\leq [D_HU]_{H\text{-Lip}} e^{-2\lambda_1^{-1}s}\pa{\|T_{\eta}(s)|D_H\tilde{f}|_H\|_{L^2(X,\nu_{\eps_n})}
+\|T_{\Omega}(s)|D_H f|_H\|_{L^2(\Omega,\nu)}}
\\
&\leq 2[D_HU]_{H\text{-Lip}}(\nu(X))^{\frac{1}{2}}\|D_H f\|_{L^\infty(\Omega,\nu;H)}
\end{align*}
where in the last line we used the contractivity of $T_\eta(t)$ and
$T_\Om(t)$ in $L^\infty$ and the fact that $\nu_{\eps_n}(X)\le \nu(X)$ for any $n\in \N$.
So there exists $M>0$ large enough so that the family in \eqref{insieme} is contained in $B(0,M)$, 
the ball of $L^2(\Omega,\nu;H)$ with center $0$ and radius $M$.

Recall that every bounded subset of $L^2(\Omega,\nu;H)$ is weakly metrisable (see \cite[Proposition 3.106]{FAB1}) and
let $\rho:B(0,M)\times B(0,M)\rightarrow\R$ be a metric such that the topology generated by $\rho$ and the weak topology
in $B(0,M)$ coincide. Now we use a diagonal argument to pass to the limit in \eqref{comm_eps_int}.
Let $n_1\in\N$ such that for every $n\geq n_1$ it holds
\[
\rho\pa{{\bf T}_{\eps_n}(t-s)(D_H^2 U\overline{\Gamma}_1(s,\cdot)), {\bf T}_{\Om}(t-s)(D_H^2 UD_HT_\Om(s)f)}\leq 1.
\]
where $\overline{\Gamma}_j(s,x)= \Gamma_{j^{-1}}(s,x)$ for any $s>0$ and $x\in X$. Now assume that $n_1,\ldots, n_k$ are already constructed and consider $n_{k+1}>n_k$ be such that for every $n\geq n_{k+1}$
\[
\rho\pa{{\bf T}_{\eps_n}(t-s)(D_H^2 U\overline{\Gamma}_{n_k}(s,\cdot)), {\bf T}_{\Om}(t-s)(D_H^2 UD_HT_\Om(s)f)}
\leq \frac{1}{2^k}.
\]
Consider now the sequence $({\bf T}_{\eps_{n_k}}(t-s)(D_H^2 U\overline{\Gamma}_{n_k}(s,\cdot)))_{k\in\N}$ and observe that
it weakly converges to ${\bf T}_{\Om}(t-s)(D_H^2 UD_HT_\Om(s)f)$ in $L^2(\Omega,\nu;H)$ as $k\to \infty$.

\noindent
{\bf Step 2.} To complete the proof, we replace $\eps$ in \eqref{comm_eps_int} by a sequence
$\eps_{m}\downarrow 0$ such that step 1 and Theorems \ref{approx theorem}, \ref{Appr-vect} apply.
Let us show that we can take the limit as $m\to \infty$.
Indeed, from Theorem \ref{approx theorem} it follows that for any $f \in L^2(\Om, \nu)$, $T_{\eps_m}(t) \tilde{f}$ weakly converges (up to a subsequence) to $T_\Omega(t)f$ in $D^{2,2}(\Omega,\nu)$ as $m \to \infty$, hence
writing $T_m, {\bf T}_m, \Phi_m, \Gamma_m$ in place of $T_{\eps_m}, {\bf T}_{\eps_m}, \Phi_{\eps_m}, \Gamma_{\eps_m}$
we obtain 
\begin{gather*}
\lim_{m \to \infty}\int_\Om \langle D_H T_{m}(t)\tilde{f}, G\rangle_H d\nu=
\int_\Om \langle D_H T_\Om(t)f, G\rangle_H d\nu
\end{gather*}
and by the analogous vector-valued result (see Theorem \ref{Appr-vect}, \eqref{comparison} and again \eqref{gra-est1})
\begin{gather*}
\lim_{m \to \infty} \int_\Om\langle e^{-tQ_\infty^{-1}}{\bf T}_{m}(t)D_H \tilde{f}, G\rangle_H d\nu=
\int_\Om \langle e^{-tQ_\infty^{-1}}{\bf T}_{\Omega}(t)D_H f, G\rangle_H d\nu.
\end{gather*}
To conclude we have to prove that the last term in the right hand side of \eqref{comm_eps_int} converges to the last term in the right hand side of \eqref{aim}.
\begin{align*}
&\abs{\int_\Om\!\langle \mathbf{T}_m(t-s)(D^2_H\Phi_m D_HT_m(s)\tilde{f})d\nu
\!-\!\!\int_\Om\!\!{\bf T}_\Omega(t-s)(D_H^2 U D_H T_\Omega(s)f),G\rangle_H d\nu}
\\
&\leq \abs{\int_\Om\langle \mathbf{T}_m(t-s)(D^2_H\Phi_m D_HT_m(s)\tilde{f})
-{\bf T}_m(t-s)(D_H^2 \Phi_m \Gamma_m(s,\cdot)),G\rangle_Hd\nu}
\\
&+ \abs{\int_\Om\langle {\bf T}_m(t-s)(D_H^2 \Phi_m \Gamma_m(s,\cdot))
-{\bf T}_m(t-s)(D_H^2 U \Gamma_m(s,\cdot)),G\rangle_Hd\nu}
\\
&+ \abs{\int_\Om\langle {\bf T}_m(t-s)(D_H^2 U\Gamma_m(s,\cdot))
-{\bf T}_\Omega(t-s)(D_H^2 U D_H T_\Omega(s)f),G\rangle_Hd\nu}
\\
&=:I_1(m)+I_2(m)+I_3(m)
\end{align*}
Let us estimate $I_1$. Using that $\Phi_m\equiv U$ on $\Om$ for every $m\in\N$, formula \eqref{comparison} and the invariance property of $T_m$ with respect to $\nu_m:=\nu_{\eps_m}$ we have
\begin{align}\label{est-i1}
&I_1(m) \leq  \int_\Om \Big|\mathbf{T}_m(t-s)(D^2_H\Phi_m D_HT_m(s)\tilde{f})
-{\bf T}_m(t-s)(D_H^2 \Phi_m \Gamma_m(s,\cdot))\Big|_H\abs{G}_Hd\nu
\notag\\
&\leq  (\nu(X))^{\frac{1}{2}}\norm{G}_\infty\!
\pa{\!\int_\Om \abs{\mathbf{T}_m(t-s)(D^2_H\Phi_m D_HT_m(s)\tilde{f}-D_H^2 \Phi_m \Gamma_m(s,\cdot))}^2_Hd\nu\!}^{\frac{1}{2}}\notag\\
&\leq  (\nu(X))^{\frac{1}{2}}\norm{G}_\infty\!
\pa{\!\int_\Om T_m(t-s)\abs{D^2_H\Phi_m D_HT_m(s)\tilde{f}-D_H^2 \Phi_m \Gamma_m(s,\cdot)}^2_Hd\nu_m\!}^{\frac{1}{2}}
\notag\\
&\leq  (\nu(X))^{\frac{1}{2}}\norm{G}_\infty\!
\pa{\int_X T_m(t-s)\abs{D^2_H\Phi_m D_HT_m(s)\tilde{f}-D_H^2 \Phi_m \Gamma_m(s,\cdot)}^2_Hd\nu_m\!}^{\frac{1}{2}}
\notag\\
&=  (\nu(X))^{\frac{1}{2}}\norm{G}_\infty
\pa{\int_X \abs{(D^2_H\Phi_m D_HT_m(s)\tilde{f})-(D_H^2 \Phi_m \Gamma_m(s,\cdot))}^2_Hd\nu_m}^{\frac{1}{2}}
\notag\\
&\leq (\nu(X))^{\frac{1}{2}}\norm{G}_\infty
\pa{\int_\Omega \|D_H^2 U\|_{\mathcal{H}_2}^2 \abs{D_H T_m(s)\tilde{f}
-D_H T_\Omega(s)f}^2_Hd\nu}^{\frac{1}{2}}
\notag\\
&\leq [D_H U]_{H\text{-Lip}}(\nu(X))^{\frac{1}{2}}\norm{G}_\infty
\pa{\int_\Omega \abs{D_H T_m(s)\tilde{f}-D_H T_\Omega(s)f}^2_Hd\nu}^{\frac{1}{2}}.
\end{align}
The right hand side of \eqref{est-i1} converges to zero as $m\to\infty$: indeed, $D_H T_m(s)\tilde{f}$
converges pointwise $\nu$-almost everywhere in $\Omega$ to $D_H T_\Omega(s)f$. Furthermore, by Proposition \ref{main_properties} we have that
$\nu$-a.e. in $\Omega$
\begin{align*}
&|D_H T_{\eps_m}(s)\tilde{f}-D_HT_\Omega(s)f|^2_H
\leq 2\left(|D_H T_{\eps_m}(s)\tilde{f}|^2_H+|D_HT_\Omega(s)f|^2_H\right)
\\
&\leq 2e^{-2\lambda_1^{-1}s}\pa{T_{\eps_m}(s)|D_H \tilde{f}|_H^2+T_{\Om}(s)|D_H f|_H^2}
\leq 2\norm{D_H f}_{L^\infty(\Omega,\nu;H)}.
\end{align*}
So by the dominated convergence theorem we get that $I_1(m)$ vanishes as $m\to \infty$.
Now, using similar arguments we can estimate $I_2(m)$ as follows
\begin{align}\label{maj}
&I_2(m)\leq (\nu(X))^{\frac{1}{2}}\norm{G}_\infty
\Bigl(\int_X \abs{(D_H^2 \Phi_m \Gamma_m(s,\cdot))-(D_H^2 U \Gamma_m(s,\cdot))}^2_Hd\nu_m\Bigr)^{\frac{1}{2}}
\notag \\
&\leq (\nu(X))^{\frac{1}{2}}\norm{G}_\infty
\Bigl(\int_X \|D_H^2 \Phi_m-D_H^2U\|_{\mathcal{H}_2}^2 |\Gamma_m(s,\cdot)|_H^2d\nu_m\Bigr)^{\frac{1}{2}}
\notag\\
&\leq (\nu(X))^{\frac{1}{2}}\norm{G}_\infty
\Bigl(\int_X \|D_H^2 \Phi_m-D_H^2U\|_{\mathcal{H}_2}^2 |(D_HT_m(s)\tilde{f})\chi_{X\smallsetminus\Omega}
\notag\\
&\phantom{\leq (\nu(X))^{\frac{1}{2}}\norm{G}_\infty\Bigl(\int_X}
+(D_HT_\Omega(s)f)\chi_\Omega|_H^2d\nu_m\Bigr)^{\frac{1}{2}}
\notag\\
&\leq (2\nu(X))^{\frac{1}{2}}\norm{G}_\infty
\Bigl(\int_X \|D_H^2 \Phi_m-D_H^2U\|_{\mathcal{H}_2}^2
\bigl(|D_HT_m(s)\tilde{f}|^2_H
\notag \\
&\phantom{\leq (2\nu(X))^{\frac{1}{2}}\norm{G}_\infty\Bigl(\int_X}
+|D_HT_\Omega(s)f|_H^2\chi_\Omega\bigr)d\nu_m\Bigr)^{\frac{1}{2}}
\notag\\
&\leq(2\nu(X))^{\frac{1}{2}}\norm{G}_\infty
\Bigl(\int_X \|D_H^2 \Phi_m-D_H^2U\|_{\mathcal{H}_2}^2
\bigl(T_m(s)|D_H\tilde{f}|^2_H
\notag \\
&\phantom{\leq(2\nu(X))^{\frac{1}{2}}\norm{G}_\infty\Bigl(\int_X}
+(T_\Omega(s)|D_Hf|_H^2)\chi_\Omega\bigr)d\nu_m\Bigr)^{\frac{1}{2}}
\notag\\
&\leq 2(\nu(X))^{\frac{1}{2}}\norm{D_Hf}_{L^\infty(\Omega,\nu;H)}\norm{G}_\infty
\Bigl(\int_X \|D_H^2 \Phi_m-D_H^2U\|_{\mathcal{H}_2}^2 d\nu_m\Bigr)^{\frac{1}{2}}
\notag\\
&= 2(\nu(X))^{\frac{1}{2}}\norm{D_Hf}_{L^\infty(\Omega,\nu;H)}\norm{G}_\infty\!
\Bigl(\!\int_X\!\! \frac{1}{4\eps_m}
\|D_H^2 d_\Omega^2\|_{\mathcal{H}_2}^2 e^{-U-\frac{1}{2\eps_m}d^2_\Omega}d\gamma\!\Bigr)^{\frac{1}{2}}\!\!.
\end{align}
Now observe that the right hand side of \eqref{maj} vanishes as $m\to \infty$. Indeed the function $\frac{1}{4\eps_m}\|D_H^2 d_\Omega^2\|_{\mathcal{H}_2}^2 e^{-U-\frac{1}{2\eps_m}d^2_\Omega}$ identically vanishes in $\Om$ and converges pointwise to $0$ $\nu$-almost everywhere in $X\setminus \Omega$ as $m\to \infty$. Furthermore let observe that the function
$(0,\infty)\ni \eps \mapsto R(\eps):=\frac{1}{4\eps^2}\abs{D^2_Hd_\Omega^2}^2e^{-\frac{1}{2{\eps}}d_\Omega^2}$ attains its maximum in $\eps=d_\Omega^2/4$ where it equals to $4d_\Omega^{-4}\|D_H^2 d_\Omega^2|_{\mathcal{H}_2}^2$. Thus, using Hypothesis \ref{hyp_d} and applying the dominated convergence theorem we infer that also $I_2(m)$ converges to zero as $m$ goes to infinity.

Finally $I_3(m)$ converges to zero as $m$ goes to infinity thanks to step 1 and this concludes the proof.
\end{proof}

\begin{corollary}\label{s1s2}
Assume Hypotheses $\ref{hyp_base}$, $\ref{ipotesi peso}$, $\ref{ipo_convex}$ and $\ref{hyp_d}$ hold true. For any $t>0$ and $p>1$ there exist two operators ${\bf S}_1(t): L^p(\Omega,\nu;H)\ra L^1(\Omega,\nu;H)$ and $S_2(t): L^p(\Omega,\nu)\ra L^1(\Omega,\nu;H)$ such that for every continuous and $H$-differentiable function $\varphi:\Omega\ra\R$ with 
$H$-Lipschitz gradient
\begin{gather*}
D_HT_\Omega(t)\varphi={\bf S}_1(t)D_H\varphi+S_2(t)\varphi.
\end{gather*}
Moreover, the adjoint operator $({\bf S}_1(t))^*$ maps ${\rm {\rm Lip}}_c(\Omega;H)$ into $L^{\infty}(\Omega,\nu;H)$ and
$\norm{{\bf S}_1^*(t)F}_\infty\leq C_1(t)\norm{F}_\infty$ for any $F\in {\rm {\rm Lip}}_c(\Omega;H)$ with
$C_1(t)\to 1$ as $t\to 0$ and the norm $C_2(t):=\norm{S_2(t)}_{\mathcal{L}(L^p,L^1)}\to 0$ as $t\to 0$.
\end{corollary}

\begin{proof}
Setting ${\bf S}_1(t):=e^{-tQ_\infty^{-1}}{\bf T}_\Omega(t)$,
\cite[Proposition 1.10]{AFP18} yields that ${\bf S}_1^*(t)={\bf T}_\Omega(t)e^{-tQ_\infty^{-1}}$ maps
${\rm Lip}_c(\Omega;H)$ into $L^{\infty}(\Omega,\nu;H)$ and
\begin{align}\label{1}
\norm{{\bf S}_1^*(t)F}_\infty&=\|(e^{-tQ_\infty^{-1}}{\bf T}_\Omega(t))^*F\|_\infty
\notag \\
&=\|{\bf T}_\Omega(t)e^{-tQ_\infty^{-1}}F\|_\infty\leq |e^{-tQ_\infty^{-1}}|_{{\mathcal{L}}(H)}\norm{F}_\infty.
\end{align}
Moreover, setting $S_2(t):=-\int_0^te^{(s-t)Q_\infty^{-1}}{\bf T}_\Om(t-s)(D_H^2 U D_H T_\Om))ds$,
by the contractivity of $\mathbf{T}_\Omega(t)$ in $L^1(\Omega,\nu; H)$, \eqref{gra-est1}, Hypothesis \ref{ipotesi peso},
the contractivity of $e^{-tQ^{-1}_\infty}$ in $H$, estimate \eqref{comparison1} and the invariance property of
$T_\Omega(t)$, we get
\begin{align}\label{2}
&\norm{{\bf S}_2(t) \varphi}_{L^1(\Omega,\nu;H)}
\le \int_0^t\int_\Omega\abs{e^{(s-t)Q_\infty^{-1}}{\bf T}_\Omega(t-s)(D^2_HUD_HT_\Omega(s)\varphi)}_Hd\nu ds
\notag\\
&\le \int_0^t\left|e^{(s-t)Q_\infty^{-1}}\right|_{{\mathcal{L}}(H)}
\int_\Omega\abs{{\bf T}_\Omega(t-s)(D^2_HUD_HT_\Omega(s)\varphi)}_Hd\nu ds
\notag\\
&\le \int_0^t \int_\Omega\abs{{\bf T}_\Omega(t-s)(D^2_HUD_HT_\Omega(s)\varphi)}_Hd\nu ds
\notag\\
&\le \int_0^t \int_\Omega\abs{D^2_HUD_HT_\Omega(s)\varphi}_Hd\nu ds
\le \int_0^t \int_\Omega \|D^2_HU\|_{\mathcal{H}_2}\abs{D_HT_\Omega(s)\varphi}_Hd\nu ds
\notag\\
&\le \int_0^t \pa{\int_\Omega \|D^2_HU\|^{p'}_{\mathcal{H}_2}d\nu}^{\frac{1}{p'}}
\pa{\int_\Omega\abs{D_HT_\Omega(s)\varphi}^p_Hd\nu}^{\frac{1}{p}} ds
\notag\\
&\le K_p^{\frac{1}{p}}\|D^2_HU\|_{L^{p'}(X,\nu;\mathcal{H}_2)}
\int_0^t s^{-\frac{1}{2}}\pa{\int_\Omega T_\Omega(s)|\varphi|^pd\nu}^{\frac{1}{p}} ds
\notag\\
&\le K^{\frac{1}{p}}_p\|D^2_HU\|_{L^{p'}(X,\nu;\mathcal{H}_2)}\|\varphi\|_{L^p(\Omega,\nu)}
\int_0^t  s^{-\frac{1}{2}} ds
\notag \\
&= 2K^{\frac{1}{p}}_p\sqrt{t}\|D^2_HU\|_{L^{p'}(X,\nu;\mathcal{H}_2)}\|\varphi\|_{L^p(\Omega,\nu)}
\end{align}
for any $t>0$. By the assumption on $U$ we deduce that the operator $S_2(t)$ is bounded from $L^p(\Omega,\nu)$ into $L^1(\Omega,\nu;H)$ for any $t>0$. Finally, estimates \eqref{1} and \eqref{2} allow us to complete the proof.
\end{proof}

Now, we are able to prove the main result of this section.

\begin{theorem}\label{bv_dg}
Assume Hypotheses $\ref{hyp_base}$, $\ref{ipotesi peso}$, $\ref{ipo_convex}$ and $\ref{hyp_d}$ hold true and let $u\in L^2(\Om,\nu)$. The following statement are true:
\begin{enumerate}[\rm (i)]
\item if $\liminf_{t \to 0^+}\| D_H T_\Om (t)u \|_{L^1(\Om,\nu;H)}$ is finite, then $u\in BV(\Omega,\nu)$;\label{first}

\item if $u\in BV(\Omega,\nu)$, then $\limsup_{t \to 0^+}\| D_H T_\Om (t)u \|_{L^1(\Om,\nu;H)}\le |D_\nu u|(\Omega)$.\label{second}
\end{enumerate}
Hence, $u\in BV(\Omega,\nu)$ iff $\displaystyle\lim_{t\ra 0^+}\| D_H T_\Om (t)u \|_{L^1(\Om,\nu;H)}<\infty$.
In this case
\begin{equation}\label{DeG}
|D_\nu u|(\Omega)=\lim_{t\ra 0^+}\| D_H T_\Om (t)u \|_{L^1(\Om,\nu;H)}.
\end{equation}
\end{theorem}

\begin{proof}
\eqref{first} follows from the strong continuity of $T_{\Om}(t)$ in $L^1(\Om, \nu)$,
see Proposition \ref{main_properties}(i), and the lower semicontinuity of the norm \eqref{mis=var}, which imply
\[
|D_\nu u|(\Omega)\le\liminf_{t\to 0^+}\int_\Omega |D_H T_\Omega(t)u|_Hd\nu.
\]
To prove \eqref{second} we write the $L^1$-norm of the gradient of $T_\Om(t)u$ by duality, as
\[
\|D_H T_\Omega(t)u\|_{L^1(\Om, \nu;H)}=
\sup\left\{\int_\Om \langle D_H T_\Omega(t)u,F \rangle_H d\nu:
\begin{array}{l}
F\in {\rm Lip}_c(\Om;H),\\
\|F\|_\infty\le 1
\end{array}\right\}.
\]
Taking into account that, for any $F\in {\rm Lip}_b(\Om;H)$ we get
\begin{align*}
&\int_\Om \langle D_H T_\Omega(t)u,F \rangle_H d\nu
= \int_\Om u  (D_H T_\Omega(t))^*F  d\nu
\le \int_\Om u  \Big({\bf S}_1(t)D_H+S_2(t)\Big)^*F  d\nu
\\
&= \int_\Om u  (D_H^*{\bf S}_1(t)^*F +S_2(t)^*F)  d\nu
\le |D_\nu u|(\Omega)\|{\bf S}_1(t)^*F\|_{\infty}+ \|u\|_{L^2(\Om,\nu)}\|S_2(t)^*F\|_{L^2(\Om,\nu)}
\\
& \le (|D_\nu u|(\Omega)C_1(t)+C_2(t))\|F\|_\infty
\end{align*}
we deduce that
\begin{equation}\label{gra-var}
\|D_H T_\Omega(t)u\|_{L^1(\Om, \nu;H)}\le C_1(t)|D_\nu u|(\Omega)+C_2(t)
\end{equation}
for any $t>0$ where $C_i$ ($i=1,2$) are the positive functions in Corollary \ref{s1s2}. Thus, taking the
limsup as $t\to 0^+$ in \eqref{gra-var} we get
\begin{equation*}
\limsup_{t\to 0^+}\int_\Omega |D_H T_\Omega(t)u|_Hd\nu\le |D_\nu u|(\Omega).
\end{equation*}
and the proof is complete.
\end{proof}

It follows from Theorem \ref{bv_dg} that functions in $BV(\Om,\nu)$ may be approximated
{\em in variation} by smooth functions. This result was already known in infinite dimension when
$\Om=X$ and $T_\Om(t)$ is the Ornstein--Uhlenbeck semigroup and in a convex set, see \cite{LMP},
where the approximation is based on finite dimensional reductions of the semigroup generated by the
Neumann Ornstein--Uhlenbeck operator in $\Om$.

\begin{proposition}\label{densita}
Under Hypotheses $\ref{hyp_base}$, $\ref{ipotesi peso}$, $\ref{ipo_convex}$ and $\ref{hyp_d}$, for any
$f\in BV(\Om,\nu)$ there exists a sequence $(f_n)_{n\in\N}\subseteq D^{1,2}(\Om,\nu)$
such that
\begin{equation}\label{approx_var}
{\rm (i)}\,\lim_{n\to \infty}\|f_n-f\|_{L^2(\Om,\nu)}=0\ \ {\rm and}\ \ 
{\rm(ii)}\,\lim_{n\to \infty}\int_{\Om}|D_H f_n|_Hd\nu=|D_\nu f|(\Om).
\end{equation}
If $C\subseteq \Om$ is closed and $|D_\nu f|(\partial C)=0$ then
$\displaystyle |D_\nu f|(C)=\lim_{n\to \infty}\int_{C}|D_H f_n|_Hd\nu.$
\end{proposition}

\begin{proof}
Consider the semigroup $T_\Om(t)$ generated in $L^2(\Om,\nu)$ by the operator
$L_\Om$ defined in \eqref{defn_LOm}. It is known that for any $f\in L^2(\Om, \nu)$ the function $T_\Om(t)f$ belongs
to $D^{1,2}(\Om,\nu)$ for any $t>0$ and by the strong continuity of $T_\Om(t)$, $T_\Om(t)f$ converges to $f$ in
$L^2(\Om,\nu)$ as $t\to 0^+$. Moreover, Theorem \ref{bv_dg} implies that $\|D_H T_\Om(t)f\|_{L^1(\Om,\nu;H)}$ converges to
$|D_\nu f|(\Om)$ as $t \to 0^+$. Thus \eqref{approx_var} is proved.
To complete the proof let us observe that, by the lower semicontinuity of the total variation, for any open set $A\subseteq \Om$
\begin{equation}\label{aperti}
|D_\nu f|(A)\le \liminf_{n\to \infty}\int_{A}|D_H f_n|_Hd\nu
\end{equation}
(see \cite[Corollary 2.5]{LMP}). Analogously we deduce that
\begin{equation}\label{chiusi}|D_\nu f|(C)\ge \limsup_{n\to \infty}\int_{C}|D_H f_n|_Hd\nu
\end{equation}
for any closed subset $C\subseteq \Om$. Indeed, by \eqref{aperti} we obtain
\begin{align*}
&|D_\nu f|(\Om)-|D_\nu f|(C)=|D_\nu f|(\Om\setminus C)
\le \liminf_{n\to \infty}\int_{\Om\setminus C}|D_H f_n|_Hd\nu
\\
&= \liminf_{n\to \infty}\left(\int_\Om|D_H f_n|_Hd\nu-\int_C|D_H f_n|_Hd\nu\right)
\\
&= \lim_{n\to \infty}\int_{\Om}|D_H f_n|_Hd\nu-\limsup_{n \to \infty}\int_C|D_H f_n|_Hd\nu
\end{align*}
whence, using \eqref{approx_var}(ii), estimate \eqref{chiusi} follows. Now, using estimates \eqref{aperti},
\eqref{chiusi} and the fact that $|D_\nu f|(\partial C)=0$ we obtain
\begin{align}\label{ahahaha}
|D_\nu f|(C)=|D_\nu f|(\mathring{C})&\le  \liminf_{n\to \infty}\int_{\mathring{C}}|D_H f_n|_Hd\nu\notag\\
&\le \limsup_{n\to \infty}\int_{\mathring{C}}|D_H f_n|_Hd\nu\le |D_\nu f|(C),
\end{align}
where $\mathring{C}$ denotes the interior of $C$. Estimate \eqref{ahahaha} yields the claim.
\end{proof}

We conclude this section showing that estimate \eqref{gra-est} and the previous approximation result allow to improve
estimate \eqref{gra-var} obtaining \eqref{est}.
\begin{theorem}
Under Hypotheses $\ref{hyp_base}$, $\ref{ipotesi peso}$, $\ref{ipo_convex}$ and $\ref{hyp_d}$, if
$f\in BV(\Om, \nu)$ then
\begin{equation}\label{est}
\int_\Om |D_H T_\Om(t)f|_H d\nu\le e^{-\lambda_1^{-1} t}|D_\nu f|(\Om),\qquad\;\, t>0,
\end{equation}
$\lambda_1$ being the maximum eigenvalue of the covariance operator $Q_\infty$, see \eqref{qinfty}. Moreover,
for any open set $A\subset \Om$ with $\overline{A}\subset \Omega$,
\begin{equation*}
\lim_{t \to 0^+}\int_A |D_H T_\Om(t)f|_H d\nu=|D_\nu f|(A).
\end{equation*}
\end{theorem}

\begin{proof}
Let $f\in BV(\Om,\nu)$ and let $(f_n)_{n\in\N}\in D^{1,2}(\Om,\nu)$ be the sequence given by Proposition \ref{densita}.
By the contractivity of $T_\Om(t)$ we deduce that $T_\Om(t)f_n$ converges to $T_\Om(t)f$ in $L^2(\Om, \nu)$ as
$n\to\infty$. This fact, together with the lower semicontinuity of the total variation, \eqref{gra-est} and \eqref{invariance} yield
\begin{align*}
&\int_{\Om}|D_H T_\Om(t)f|_Hd\nu\le
\liminf_{n\to \infty}\int_{\Om}|D_H T_\Om(t)f_n|_Hd\nu
\\
&\leq e^{-\lambda_1^{-1} t}\liminf_{n\to \infty}\int_\Om T_\Omega(t)|D_H f_n|_H d\nu
\le e^{-\lambda_1^{-1} t}\lim_{n\to \infty}\int_\Om |D_H f_n|_H d\nu
\\
&= e^{-\lambda_1^{-1} t}|D_\nu f|(\Om)
\end{align*}
whence \eqref{est} is proved. The last assertion follows immediately from Proposition \ref{densita} taking into account that $T_\Om(t)f$ satisfies \eqref{approx_var}.
\end{proof}

\section{Sets of finite perimeter in $\Om$}

This section is devoted to provide some sufficient and necessary conditions in order that a Borel set $E \subseteq X$
have finite perimeter in $\Om$. We consider also the case of $BV(\Om,\nu)$ functions and $\Om=X$.
There are three semigroups involved: beside $T_\Omega(t)$, we consider the Ornstein-Uhlenbeck semigroup
$S(t)$ generated in $L^2(X, \gamma)$ by the realisation of the operator
\begin{equation*}
L_{OU}\varphi = {\rm Tr}(D^2_H \varphi)-\sum_{i=1}^{\infty}\lambda_i^{-1}\langle x,e_i\rangle D_i\varphi
\quad \varphi \in  \F C^2_b(X)
\end{equation*}
and the semigroup $T(t)$ generated in $L^2(X,\nu)$ by the realisation of the operator
\begin{equation}\label{op-L}
L\varphi = L_{OU}\varphi
%{\rm Tr}(D^2_H \varphi)-\sum_{i=1}^{\infty}\lambda_i^{-1}\langle x,e_i\rangle D_i\varphi
-\langle D_HU,D_H\varphi\rangle_H ,
\quad \varphi\in \F C^2_b(X) .
\end{equation}
Recall that $S(t)$ admits a pointwise representation by means of the Mehler formula \eqref{Mehler}.

\begin{theorem}\label{thm_nec}
Assume Hypotheses $\ref{hyp_base}$, $\ref{ipotesi peso}$, $\ref{ipo_convex}$ and $\ref{hyp_d}$ hold true and let $E\subseteq X$ be a Borel set such that $P_\nu(E,\Om )<\infty$. Then
\begin{equation}\label{cond-nec}
\limsup_{t \to 0^+}\frac{1}{\sqrt{t}}\|T_\Om(t)\chi_E-\chi_E\|_{L^1(\Om, \nu)}<\infty.
\end{equation}
More precisely
\begin{equation}\label{est*}
\|T_\Om(t)u-u\|_{L^1(\Om, \nu)}\le 2 \sqrt{K_2t}|D_\nu u|(\Omega)
\end{equation}
for any $u \in BV(\Om, \nu)$ and $t>0$ where $K_2$ is the constant in \eqref{gra-est1}.
%Further, if $\Om=X$, estimate \eqref{est*} can be replaced by the formula
%\begin{equation}\label{est-X}
%\| T(t)u-u\|_{L^1(X, \nu)}\le c_t|D_\nu u|(X)+o(\sqrt{t}), \qquad \text{ as } t\to 0^+;
%\end{equation}
%which holds true for every $u \in  BV(X,\nu)$ and $t>0$ with $c_t\simeq C \sqrt{t}$ as $t \to 0^+$ for some positive constant $C$.
\end{theorem}

\begin{proof}
Clearly, once estimate \eqref{est*} is proved, \eqref{cond-nec} follows at once choosing $u=\chi_E$. Thus, let us prove \eqref{est*}. To this aim, we consider $g\in L^\infty(\Om, \nu)$ and assume first that $u \in D^{1,2}(\Om,\nu)$. By the self-adjointness of the operators $L_\Om T_\Om(s)$ in $L^2(\Om,\nu)$ we have
\begin{align*}
&\int_\Om g(T_\Om(t)u-u)d\nu=
\int_\Om g\int_0^t \frac{d}{ds}T_\Om(s)u ds d\nu
= \int_0^t \int_\Om g (L_\Om T_\Om(s)u) d\nu ds
\notag\\
&= \int_0^t \int_\Om (L_\Om T_\Om(s)g) u  d\nu ds
= -\int_0^t \int_\Om \langle D_H T_\Om(s)g, D_H u\rangle_H d\nu ds.
\end{align*}
The Cauchy--Schwarz inequality and \eqref{gra-est1} yield
\begin{align*}
\int_\Om g(T_\Om(t)u-u)d\nu &\leq \int_0^t \int_\Om |D_H T_\Om(s)g|_H |D_H u|_H  d\nu ds\\
&=\int_0^t \int_\Om (|D_H T_\Om(s)g|^2_H)^{\frac{1}{2}} |D_H u|_H  d\nu ds\\
&\leq \sqrt{K_2}\int_0^t s^{-\frac{1}{2}}\int_\Om (T_\Om(s)|g|^2)^{\frac{1}{2}} |D_H u|_H  d\nu ds.
\end{align*}
From the contractivity of $T_\Om(t)$ in $L^\infty(\Omega,\nu)$, for any $t>0$ we deduce
\begin{align}\label{corona}
\int_\Om g(T_\Om(t)u-u)d\nu &\leq \sqrt{K_2} \norm{g}_{L^\infty(\Om,\nu)}
\int_0^t s^{-\frac{1}{2}}\int_\Om|D_H u|_H  d\nu ds\notag\\
&=2\sqrt{K_2 t}\norm{g}_{L^\infty(\Om,\nu)}\int_\Om  |D_H u|_H  d\nu .
\end{align}
For $u\in BV(\Om, \nu)$, from Proposition \ref{densita} we get a sequence $u_n\in D^{1,2}(\Om, \nu)$ converging
to $u$ in $L^2(\Om, \nu)$ with $\lim_{n \to \infty}\int_\Om |Du_n|_H d\nu=|D_\nu u|(\Om)$. Thus, putting
$u_n$ in place of $u$ in \eqref{corona} and letting $n \to \infty$ we get
\begin{align*}
\int_\Om g(T_\Om(t)u-u)d\nu &\leq 2\sqrt{K_2t}\norm{g}_{L^\infty(\Om,\nu)} |D_\nu u|(\Om), \qquad\;\,g\in L^\infty(\Om,\nu).
\end{align*}
Finally, taking the supremum on the $g\in L^\infty(\Omega,\nu)$ with $\|g\|_\infty\leq 1$ we obtain
\begin{align*}
\int_\Om |T_\Om(t)u-u|d\nu &\leq 2\sqrt{K_2t} |D_\nu u|(\Om)
\end{align*}
whence \eqref{est*} follows.
\end{proof}

\begin{remark}{\rm
Note that condition \eqref{cond-nec} is equivalent to
\begin{equation*}
\limsup_{t \to 0^+}\frac{1}{\sqrt t}\int_{E^c\cap \Om} (T_\Om(t)\chi_E)d\nu<\infty.
\end{equation*}
Indeed, $|T_\Om(t)\chi_E-\chi_E|= (\chi_E-T_\Om(t)\chi_E)\chi_E+ (T_\Om(t)\chi_E-\chi_E)\chi_{E^c}$.
The invariance of $T_\Om(t)$ with respect to $\nu$ in $\Omega$ yields
\[\int_\Om (\chi_E-T_\Om(t)\chi_E)\chi_Ed\nu=\int_\Om (T_\Om(t)\chi_E- \chi_ET_\Om(t)\chi_E)d\nu=\int_\Om \chi_{E^c}T_\Om(t)\chi_E d\nu.\]
Consequently,
\begin{align*}
\int_\Om |T_\Om(t)\chi_E-\chi_E|d\nu = 2 \int_{\Om\cap E^c}T_\Om(t)\chi_E d\nu
\end{align*}}
\end{remark}

Now, we prove a quasi converse of Theorem \ref{thm_nec}. We start with a preliminary result for
bounded functions.

\begin{proposition}\label{cond_suff}
Under Hypotheses $\ref{hyp_base}$, $\ref{ipotesi peso}$, $\ref{ipo_convex}$ and $\ref{hyp_d}$, let
$u\in L^\infty(X,\nu)$ be such that
\begin{equation}\label{cond_suff*}
\liminf_{t \to 0^+}\frac{1}{\sqrt{t}}\int_\Om\int_X|u(e^{-t}x+\sqrt{1-e^{-2t}y})-u(x)|d\gamma(y)d\nu(x) =
C < \infty.
\end{equation}
Then $u \in BV(\Om,\nu)$ and $|D_\nu u|(\Om)\le C\|Q_\infty^{1/2}\|_{\mathcal{L}(X)}\sqrt{\pi}/2$.
\end{proposition}

\begin{proof}
We divide the proof in two steps. \\
{\bf Step 1.}
Here we prove that for any $v \in C^1_b(X)$, it holds that
\begin{align}\label{aim1}
\lim_{t \to 0^+}\frac{1}{\sqrt{t}}\int_\Om\int_X&|v(e^{-t}x+\sqrt{1-e^{-2t}y})-v(x)|d\gamma(y)d\nu(x)
\notag\\
&=\frac{2}{\sqrt{\pi}}\int_{\Om}|D v(x)|d\nu(x).
\end{align}
To this aim, we observe that
\begin{align*}
&K_v(t):=\int_\Om\int_X|v(e^{-t}x+\sqrt{1-e^{-2t}y})-v(x)|d\gamma(y)d\nu(x)\\
&=\int_\Om\int_X\left|\int_0^t\frac{d}{dr}v(e^{-r}x+\sqrt{1-e^{-2r}y})dr\right|d\gamma(y)d\nu(x)\\
&=\int_\Om\!\int_X\!\Bigl|\int_0^t\!\langle D v(e^{-r}x+\sqrt{1-e^{-2r}y}),-e^{-r}x+\frac{e^{-2r}}
{\sqrt{1-e^{-2r}}}y\rangle dr\Bigr|d\gamma(y)d\nu(x)
\\
& \le \int_0^t \frac{e^{-r}}{\sqrt{1-e^{-2r}}} \cdot
\\
&\quad\cdot\int_\Om\int_X \abs{\gen{D v(e^{-r}x+\sqrt{1-e^{-2r}y}),-\sqrt{1-e^{-2r}}x+e^{-r}y}}d\gamma(y)d\nu(x).
\end{align*}
Now, for $r$ fixed we perform the ``Gaussian rotation''
\begin{equation*}
(x,y)\mapsto R_r(x,y):=(e^{-r}x+\sqrt{1-e^{-2r}y},-\sqrt{1-e^{-2r}}x+e^{-r}y)=:(u,w)
\end{equation*}
to get, thanks to the invariance of $\gamma$ under $R$, 
\begin{align*}
K_v(t)&\le\int_0^t \frac{e^{-r}}{\sqrt{1-e^{-2r}}}\int_X\int_X |\gen{D v(u),w}|\cdot
\\
&\cdot\chi_{\Om}(e^{-r}u-
\sqrt{1- e^{-2r}}w)e^{-U(e^{-r}u-\sqrt{1- e^{-2r}}w)}d\gamma(u)d\gamma(w)\\
&=:\int_X\int_X f_v(t,u,w)d\gamma(w)d\gamma(u).
\end{align*}
We claim that
\[
\lim_{t \to 0^+}\frac{1}{\sqrt{t}}\int_X\int_X f_v(t,u,w)d\gamma(w)d\gamma(u)=C\int_\Om|D v|d\nu
\]
Indeed, by the convexity of $U$ there exist $z\in X$ and $a\in\R$ such that $U(x)\geq \gen{x,z}+a$,
hence
\begin{align*}
&\frac{1}{\sqrt{t}}f_v(t,u,w)\le
\frac{1}{\sqrt{t}}|D v(u)||w|e^{|\gen{z,u}|+|\gen{z,w}|+|a|}\int_{0}^t\frac{1}{\sqrt{2r}}dr
\\
&=\sqrt{2}|D v(u)||w|e^{|\gen{z,u}|+|\gen{z,w}|+|a|}\in L^1(X\times X, \gamma \otimes\gamma),
\qquad\;\, t \in (0,1)
\end{align*}
and, using De L'H\^opital's rule, for almost every $(u,w)\in\Om\times X$
\begin{align*}
\lim_{t \to 0^+}\frac{f_v(t,u,w)}{\sqrt{t}}=\sqrt{2}\chi_{\Om}(u)\gen{D v(u),w}e^{-U(u)}.
\end{align*}
So by the dominated convergence theorem we obtain
\begin{align}\label{mah1}
&\limsup_{t\ra0^+}\frac{1}{\sqrt{t}}\int_\Om\int_X|v(e^{-t}x+\sqrt{1-e^{-2t}y})-v(x)|d\gamma(y)d\nu(x)
\notag\\
&\leq \frac{2}{\sqrt{\pi}}\int_\Om|D v(u)|d\nu(u)
\end{align}
where we used that $\int_X|\gen{D v(u),w}|d\gamma(w) = \sqrt{2/\pi}|D v(u)|$. Indeed, using the factorisation $\gamma = \gamma_1\otimes\gamma^\perp$, where $\gamma_1$ is
the 1-dimensional standard Gaussian measure on $E={\rm span}Dv(u)$, we get
\begin{align*}
&\int_X|\gen{D v(u),w}|d\gamma(w) = 2 \int_{\{w:\gen{D v(u),w}>0\}} \gen{D v(u),w} d\gamma(w)
\\
&=2 |D v(u)| \int_{E^\perp}\int_0^\infty t d\gamma_1(t)d\gamma(w') = \sqrt{2/\pi}|D v(u)|.
\end{align*}
To conclude, consider the family of linear functionals $L_t: C_b(X\times X)\ra\R$, $t \in (0,1)$
\[
L_t\varphi=\frac{1}{\sqrt{t}}\int_\Om\int_X\varphi(x,y)(v(e^{-t}x+\sqrt{1-e^{-2t}}y)-v(x))d\gamma(y)d\nu(x).
\]
By \eqref{mah1} we get $\limsup_{t\ra 0^+}\norm{L_t}\leq 2(\sqrt{\pi})^{-1}\|D v\|_{L^1(\Om,\nu)}$
and arguing as above
\begin{align*}
\lim_{t\ra 0^+}L_t\varphi=\sqrt{2}\int_\Om\int_X\varphi(x,y)\gen{D u(x),y}d\gamma(y)d\nu(x)=:L_0\varphi
\end{align*}
So $L_t$ weakly$^*$ converges to $L_0$ as $t \to 0^+$ and, by lower semicontinuity of the norm
we get \eqref{aim1}:
\[
\norm{L_0}=\frac{2}{\sqrt{\pi}}\int_\Om|D v(x)|d\nu(x)\leq\liminf_{t\ra 0^+}\norm{L_t}
\leq\limsup_{t\ra0^+}\norm{L_t}\leq \frac{2}{\sqrt{\pi}}\int_\Om|D v(x)|d\nu(x).
\]
{\bf Step 2.}
For $u\in L^\infty(X, \nu)$, let $(u_j)_{j\in\N}\subseteq C^1_b(X)$ be such that $u_j\to u$ in $L^2(X, \nu)$,
almost everywhere in $X$ and satisfying \eqref{cond_suff*} (thanks to the dominated convergence theorem). Using \eqref{aim1},
\eqref{invariance} and \eqref{gra-est} we have
\begin{align}\label{simo}
&\lim_{t\to 0^+}\frac{K_{u_j}(t)}{\sqrt{t}}
=\frac{2}{\sqrt{\pi}}\int_{\Om}|D u_j|d\nu
=\frac{2}{\sqrt{\pi}\|Q_\infty^{1/2}\|_{\mathcal{L}(X)}}\int_{\Om}\|Q_\infty^{1/2}\|_{\mathcal{L}(X)}|D u_j|d\nu\notag\\
&\geq\frac{2}{\sqrt{\pi}\|Q_\infty^{1/2}\|_{\mathcal{L}(X)}}\int_{\Om}|Q_\infty^{1/2} D u_j|d\nu
=\frac{2}{\sqrt{\pi}\|Q_\infty^{1/2}\|_{\mathcal{L}(X)}}\int_{\Om}|Q_\infty D u_j|_Hd\nu
\notag \\
&=\frac{2}{\sqrt{\pi}\|Q_\infty^{1/2}\|_{\mathcal{L}(X)}}\int_{\Om}|D_H u_j|_Hd\nu
=\frac{2}{\sqrt{\pi}\|Q_\infty^{1/2}\|_{\mathcal{L}(X)}}\int_{\Om}(T_\Om(\sigma)|D_H u_j|_H)d\nu
\notag \\
&\ge \frac{2}{\sqrt{\pi}\|Q_\infty^{1/2}\|_{\mathcal{L}(X)}}e^{\sigma \lambda_1^{-1}}
\int_{\Om}|D_H T_\Om(\sigma) u_j|_Hd\nu
\end{align}
for any $\sigma \in (0,1)$. Now, since the left hand side of \eqref{simo} is uniformly bounded from above by 
the consatnt $C$, the $L^1$-norm of $D_HT_\Om(\sigma)u_j$ is bounded as well by the same constant for every 
$j \in \N$ and $\sigma \in (0,1)$, i.e.,
\[
e^{\frac{\sigma}{\lambda_1}}\int_{\Om}|D T_\Om(\sigma) u_j|d\nu
\le C\frac{\sqrt{\pi}}{2}\|Q_\infty^{1/2}\|_{\mathcal{L}(X)},\qquad j\in \N,\sigma>0.
\]
Thus, recalling that $D_H T_\Om(\sigma) u_j$ converges to $D_H T_\Om(\sigma) u$ in $L^1(\Om, \nu)$ as $j\to \infty$ (see
\eqref{gra-est1}), letting first $j \to \infty$ and then $\sigma \to 0^+$ and using formula \eqref{DeG} we get that
$|D_\nu u|(\Om)\le C\|Q_\infty^{1/2}\|_{\mathcal{L}(X)}\sqrt{\pi}/2$.
\end{proof}

The following result is a quasi converse of Theorem \ref{thm_nec}. In fact, we give a sufficient condition
to have $P_\nu(E,\Omega)<\infty$ in terms of the short-time behaviour of $T(t)$, where $T(t)$ is the semigroup generated by the operator $L$ defined in \eqref{op-L} in $L^2(X,\nu)$.  

and not of $T_\Om(t)$.

\begin{theorem}
Under Hypotheses $\ref{hyp_base}$, $\ref{ipotesi peso}$, $\ref{ipo_convex}$ and $\ref{hyp_d}$, if $E\in{\mathcal B}(X)$ and
\begin{equation}\label{cond_11}
C:=\liminf_{t \to 0^+}\frac{1}{\sqrt{t}}\|T(t)\chi_E- \chi_E\|_{L^1(\Om,\nu)}<\infty,
\end{equation}
then $P_\nu(E, \Om)\le C\|Q_\infty^{1/2}\|_{\mathcal{L}(X)}\sqrt{\pi}/2$.
\end{theorem}

\begin{proof}
Choosing $u=\chi_E$ in \eqref{cond_suff*} and observing that
$$
\int_\Om\left|\int_X f(x,y)d\gamma(y)\right|d\nu(x)=\int_\Om\int_X |f(x,y)|d\gamma(y)d\nu(x)
$$
for any $f$ with constant sign, from Proposition \ref{cond_suff} we deduce that if
\begin{equation}\label{cond_1}
L:=\liminf_{t \to 0^+}\frac{1}{\sqrt{t}}\|S(t)\chi_E- \chi_E\|_{L^1(\Om,\nu)}<\infty
\end{equation}
then $P_\nu(E, \Om)\le L\|Q_\infty^{1/2}\|_{\mathcal{L}(X)}\sqrt{\pi}/2$. Here $S(t)$ is the Ornstein--Uhlenbeck
semigroup in \eqref{Mehler}. To conclude we prove that condition \eqref{cond_1} is equivalent to \eqref{cond_11}. 
From the variation-of-constants formula we deduce
\begin{equation}\label{form_var_cost}
(T(t)g)(x) = (S(t)g)(x) -\int_0^t (S(t-\sigma)\langle D_H U, D_H T(\sigma)g\rangle_H)(x) d\sigma,
\end{equation}
for every $g\in \mathcal{F}C_b(X)$, $\nu$-a.e. $x\in X$ and any $t\ge 0$. To prove \eqref{form_var_cost}
it suffices that the map
$\sigma\mapsto S(t-\sigma)\langle D_H U, D_H T(\sigma)g\rangle_H$
belongs to $L^1((0,t))$ for any $t>0$.
To this aim, let us observe that
\begin{equation*}
\int_X\int_0^t S(t-\sigma)\langle D_H U, D_H T(\sigma)g\rangle_H d\sigma d\nu<\infty
\end{equation*}
for any $g\in\mathcal{F}C_b(X)$. Indeed, the H\"older inequality and the contractivity of $S(t)$ in $L^2(X,\gamma)$
allow us to write
\begin{align}\label{giro}
\int_X & \int_0^t S(t-\sigma)\langle D_H U, D_H T(\sigma)g\rangle_H d\sigma d\nu
\notag\\
&= \int_0^t\int_X S(t-\sigma)\langle D_H U, D_H T(\sigma)g\rangle_H d\nu d\sigma
\notag\\
& \le \int_0^t\|S(t-\sigma)\langle D_H U, D_H T(\sigma)g\rangle_H\|_{L^1(X, \nu)}  d\sigma
\notag\\
& \le \|e^{-U}\|_{L^2(X,\gamma)} \int_0^t\|S(t-\sigma)\langle D_H U, D_H T(\sigma)g\rangle_H\|_{L^2(X, \gamma)}  d\sigma
\notag\\
& \le \|e^{-U}\|_{L^2(X,\gamma)}\int_0^t\|\langle D_H U, D_H T(\sigma)g\rangle_H\|_{L^2(X, \gamma)}  d\sigma
\notag\\
& \le \sqrt{K_2}\|e^{-U}\|_{L^2(X,\gamma)}\|g\|_\infty\|D_H U\|_{L^2(X,\gamma;H)}\int_0^t \sigma^{-1/2} ds
\notag\\
&= 2\sqrt{K_2 t}\|e^{-U}\|_{L^2(X,\gamma)}\|g\|_\infty\|D_H U\|_{L^2(X,\gamma;H)}
\end{align}
where in the last line we used estimate \eqref{gra-est1} which holds true even in the case $\Om=X$ and $T_\Om(t)$  
replaced by $T(t)$. Hence, formula \eqref{form_var_cost} follows.

Now, integrating \eqref{form_var_cost} in $\Om$ with respect to $\nu$ yields
\begin{align}\label{diff}
\|S(t)\chi_E-\chi_E\|_{L^1(\Om, \nu)}- H(t)&\le \|T(t)\chi_E-\chi_E\|_{L^1(\Om, \nu)}
\notag \\
&\le \|S(t)\chi_E-\chi_E\|_{L^1(\Om, \nu)}+ H(t)
\end{align}
for any $t>0$ with
\[H(t):= \left|\int_{X}\int_0^t S(t-s)\langle D_H U,D_H T(s)\chi_E\rangle_H ds d\nu\right|, \qquad\;\, t>0.\]
Using estimate \eqref{giro} with $g=\chi_E$ we infer that $\limsup_{t\to 0^+}\frac{H(t)}{\sqrt{t}}<\infty$.
This last estimate, together with \eqref{diff}, prove that \eqref{cond_1} is equivalent to \eqref{cond_11} and the proof 
is  complete.
\end{proof}

\end{document}